\documentclass[12pt]{amsart}
\usepackage[a4paper,top=3cm,bottom=3cm,inner=2.5cm,outer=2.5cm]{geometry}

\usepackage{amsmath,amsfonts,amssymb,amsthm}
\usepackage{xparse}
\usepackage[all]{xy}
\usepackage{tikz}
\usepackage{tikz-cd}
\usepackage{multirow,enumitem,array}
\usepackage{pst-node,placeins,xspace,fix-cm}
\usepackage[nomessages]{fp}
\usepackage[colorlinks,linkcolor=blue,citecolor=violet,urlcolor=darkgray,final,hyperindex,linktoc=page,hyperfootnotes=true]{hyperref}

\newcolumntype{F}{>{$}c<{\hspace{-0.9ex}$}}
\newcolumntype{:}{>{$}m{0.8ex}<{$}}
\newcolumntype{R}{>{$}r<{$}}
\newcolumntype{C}{>{$}c<{$}}
\newcolumntype{L}{>{$}l<{$}}
\newcolumntype{N}{@{}>{$}l<{$}}
\newlength\horspace
\newcommand{\h}[1][1.0]{\hspace*{#1\horspace}}
\newlength\verspace
\renewcommand{\v}[1][1.0]{\vspace*{#1\verspace}\xspace}
\tikzset{iso/.style={draw=none,every to/.append style={edge node={node [sloped, allow upside down, auto=false]{$\cong$}}}}}
\tikzset{adjunction/.style={draw=none,every to/.append style={edge node={node [sloped, allow upside down, auto=false]{$\dashv$}}}}}
\tikzset{simeq/.style={draw=none,every to/.append style={edge node={node [sloped, allow upside down, auto=false]{$\simeq$}}}}}
\tikzset{simeqS/.style={draw=none,every to/.append style={edge node={node [sloped, allow upside down, auto=false]{$\raisebox{0.8em}{$\simeq$}$}}}}}
\tikzset{aiso/.style={simeqS,preaction={draw,->}}}
\tikzset{proarrowS/.style={draw=none,every to/.append style={edge node={node [sloped, allow upside down, auto=false]{\raisebox{1.4pt}{\small$\shortmid$}}}}}}
\tikzset{proarrow/.style={proarrowS,preaction={draw,->}}}
\tikzset{nullS/.style={draw=none,every to/.append style={edge node={node [sloped, allow upside down, auto=false]{\raisebox{-1.16 ex}{$\circ$}}}}}}
\tikzset{null/.style={nullS,preaction={draw,->}}}
\tikzset{dotdot/.style={dash pattern=on 0.25ex off 0.2ex, dash phase=0ex}}
\tikzset{RightA/.style={double distance=3.5pt,>={Implies},->},%
	triple/.style={-,preaction={draw,RightA}},%
	quadruple/.style={preaction={draw,RightA,shorten >=0pt},shorten >=1pt,-,double,double distance=0.2pt}}
\tikzset{Right/.style={double distance=1.7pt,>={Implies},->}}
\tikzset{simeqSRight/.style={draw=none,every to/.append style={edge node={node [sloped, allow upside down, auto=false]{$\raisebox{-1em}{\rotatebox{180}{$\simeq$}}$}}}}}
\tikzset{twoiso/.style={simeqSRight,preaction={draw,Right}}}

\theoremstyle{plain} 
\newtheorem{theorem}{Theorem}[section]
\newtheorem{lemma}[theorem]{Lemma}
\newtheorem{proposition}[theorem]{Proposition}
\newtheorem{corollary}[theorem]{Corollary}

\theoremstyle{definition}
\newtheorem{definition}[theorem]{Definition}
\newtheorem{remark}[theorem]{Remark}
\newtheorem{example}[theorem]{Example}

\def\nameit#1{\textrm{#1}~}
\def\thex{\nameit{Theorem}}
\def\prox{\nameit{Proposition}}
\def\corx{\nameit{Corollary}}

\def\defx{\nameit{Definition}}
\def\remx{\nameit{Remark}}
\def\exax{\nameit{Example}}

\def\dfn#1{{\itshape #1}}
\def\predfn#1{{\itshape #1}}

\NewDocumentEnvironment{cd}{s O{6} O{6} b}{%
	\IfBooleanF{#1}{\begin{equation*}}\begin{tikzcd}[row sep=#2ex,column sep=#3ex,ampersand replacement=\&]
			#4
		\end{tikzcd}\IfBooleanF{#1}{\end{equation*}}\ignorespacesafterend}{}

\newenvironment{enum}{\begin{enumerate}[label=$($\hspace{0.12ex}\roman*\hspace{0.075ex}$)$]}{\end{enumerate}}
\newenvironment{enumT}{\begin{enumerate}[label=$($\hspace{-0.1ex}\roman*\hspace{0.13ex}$)$]}{\end{enumerate}}

\newenvironment{eqD}[1]{\begin{equation}\label{#1}}{\end{equation}\ignorespacesafterend}
\newenvironment{eqD*}{\begin{equation*}}{\end{equation*}\ignorespacesafterend}

\newcommand{\C}{\mathbb{C}}

\renewcommand{\L}{\mathfrak{L}}
\newcommand{\E}{\mathcal{E}}
\newcommand{\M}{\mathcal{M}}
\newcommand{\Esq}{\mathfrak{E}}
\newcommand{\Msq}{\mathfrak{M}}
\newcommand{\N}{\mathcal{N}}

\newcommand{\0}{\mathbf{0}}

\newcommand{\Set}{\mathbf{Set}}
\newcommand{\Cat}{\mathbf{Cat}}

\NewDocumentCommand{\Alg}{t+ t' m}{
	\ensuremath{\IfBooleanT{#1}{\mathbf{Ps}\mbox{-}}{#3}\mbox{-}\mathbf{\IfBooleanT{#2}{Co}Alg}}
}

\def\:{\colon}

\def\c{\circ}
\newcommand{\iso}{\cong}
\def\phi{\varphi}
\def\eps{\varepsilon}
\newcommand{\tom}{\rightarrowtail}
\newcommand{\toe}{\twoheadrightarrow}

\newcommand{\Hom}[3][]{\operatorname{Hom}_{\mkern1mu #1}\mkern-1.5mu\left({#2},{#3}\right)}
\newcommand{\HomC}[3]{{#1}\left({#2},\h[1]{#3}\right)}

\newcommand{\id}[1]{\operatorname{id}_{#1}}
\newcommand{\Id}[1]{\operatorname{Id}_{#1}}
\newcommand{\op}{\ensuremath{^{\operatorname{op}}}}
\newcommand{\x}[1][]{\times_{#1}}

\newcommand{\aar}[2][]{\xrightarrow[#1]{#2}}
\makeatletter
\newcommand{\aR}[2][]{%
	\ext@arrow 0055{\Rightarrowfill@}{#1}{#2}}
\def\xLongrightarrowfill@{\arrowfill@\Relbar\Relbar\Longrightarrow}
\newcommand{\am}[2][]{%
	\ext@arrow 0395\xmapstofill@{#1}{#2}}
\def\xlongmapstofill@{\arrowfill@\relbar\relbar\longmapsto}
\def\xlongrightarrowfill@{\arrowfill@\relbar\relbar\longrightarrow}
\newcommand{\aarr}[2][]{%
	\ext@arrow 0099\xlongrightarrowfill@{#1}{#2}}
\newcommand{\eqq}{\DOTSB\protect\Relbar\protect\joinrel\Relbar}
\def\xeqqfill@{\arrowfill@\Relbar\Relbar\eqq}
\newcommand{\aeqq}[2][]{%
	\ext@arrow 0099\xeqqfill@{#1}{#2}}
\def\xRrightarrowfill@{\arrowfill@\equiv\equiv\Rrightarrow}
\newcommand{\aM}[2][]{\ext@arrow 0359\xRrightarrowfill@{#1}{#2}}    
\def\nullrightarrowfill@{\arrowfill@{\relbar{\circ}}\relbar\longrightarrow}
\newcommand{\anull}[2][]{
	\ext@arrow0099\nullrightarrowfill@{#1}{#2}}
\makeatother
\newcommand{\aiso}[1]{\overset{#1}{\iso}}

\newcommand{\PB}[1]{\arrow[#1,phantom,"\scalebox{1.6}{\color{black}$\lrcorner$}",very near start]}
\newcommand{\scaleu}[2][1.2]{{\scalebox{#1}{$#2$}}}
\NewDocumentCommand{\fib}{O{n} O{2.3} mmm}{%
	\begin{cd}*[#2][5]
		{#3}\ifx#1n{\arrow[d,"{\,\scaleu{#4}}"]}\else{\ifx#1i{\arrow[d,hookrightarrow,"{\,\scaleu{#4}}"]}\else{\ifx#1e{\arrow[d,equal,"{\,\scaleu{#4}}"]}\else{\ifx#1R{\arrow[d,Rightarrow,"{\,\scaleu{#4}}"]}\fi}\fi}\fi}\fi\\
		{#5}\ifx#1o{\arrow[u,"{\,\scaleu{#4}}"']}\fi
	\end{cd}\xspace
}
\newcommand{\Ar}[4][]{\arrow[#2,"{#3}"{#1},""{name=#4, anchor=center}]}
\newcommand{\Ars}[4][]{\arrow[#2,"{#3}"'{#1},""{name=#4, anchor=center}]}
\newcommand{\Arb}[6][]{\arrow[#2,"{#3}"{#1},from=#4,to=#5,shorten <= #6 em, shorten >= #6 em]}
\newcommand{\Arbs}[6][]{\arrow[#2,"{#3}"'{#1},from=#4,to=#5,shorten <= #6 em, shorten >= #6 em]}
\NewDocumentCommand{\tc}{s t+ O{7} O{30} O{} O{} O{} o}{
	\def\footc##1##2##3##4##5{%
		\FPmul\Mulresulttwo{#3}{#3}%
		\FPmul\Mulresult{0.0026}{\Mulresulttwo}%
		\IfBooleanTF{#1}{\begin{cd}*}{\begin{cd}}[#3][#3]
				{##1}\Ar[#5]{r,bend left=#4}{##3}{A}\Ars[#6]{r,bend right=#4}{##4}{B}\&{##2}
				\IfBooleanTF{#2}{\Arb[description,pos=0.49]}{\Arb}{Rightarrow #7}{\mkern1mu {##5}}{A}{B}{\IfNoValueTF{#8}{\Mulresult}{#8}}
		\end{cd}}%
		\footc}
	\NewDocumentCommand{\tcv}{s t' O{5} O{38} mmmmm}{
		\FPmul\Mulresulttwo{#3}{#3}%
		\FPmul\Mulresult{0.0026}{\Mulresulttwo}%
		\IfBooleanTF{#1}{\begin{cd}*}{\begin{cd}}[#3][#3]
				{#5}\IfBooleanTF{#2}{\Ars{d,leftarrow,bend right=#4}{#7}{A}\Ar{d,leftarrow,bend left=#4}{#8}{B}}{\Ars{d,bend right=#4}{#7}{A}\Ar{d,bend left=#4}{#8}{B}}\\{#6}
				\Arb{Rightarrow}{#9}{A}{B}{\Mulresult}
		\end{cd}}
		\NewDocumentCommand{\sq}{s O{n} O{6} O{6} O{} O{2.7} O{2.2} O{0.5} O{n}}{%
			\def\foosq##1##2##3##4##5##6##7##8{%
				\IfBooleanTF{#1}{\begin{cd}*}{\begin{cd}}[#3][#4]
						{##1}\ifx#2p{\PB{rd}}\fi\arrow[r,"{##5}"]\ifx#9l{\arrow[d,equal,"{##6}"']}\else{\arrow[d,"{##6}"']}\fi\&{##2}\ifx#9r{\arrow[d,equal,"{##7}"]}\else{\arrow[d,"{##7}"]}\fi\ifx#2l{\arrow[ld,Rightarrow,shorten <=#6ex,shorten >=#7ex,"{#5}"{pos=#8}]}\fi \ifx#2i{\arrow[ld,twoiso,shorten <=#6ex,shorten >=#7ex,"{#5}"{pos=#8}]}\fi\\
						{##3}\ifx#9d{\arrow[r,equal,"{##8}"']}\else{\arrow[r,"{##8}"']}\fi\ifx#2o{\arrow[ur,Rightarrow,shorten <=#6ex,shorten >=#7ex,"{#5}"{pos=#8}]}\fi\&{##4}
				\end{cd}}%
				\foosq}
			\NewDocumentCommand{\tr}{s O{4.5} O{6.5} O{0} O{0} O{n} O{0} O{} O{0}}{%
				\def\footr##1##2##3##4##5##6{%
					\IfBooleanTF{#1}{\begin{cd}*}{\begin{cd}}[#3][#2]
							{##1}\arrow[rr,"{##4}"]
							\Ars[inner sep =0.2ex]{dr}{##5}{A}\&[#4ex]\&[#5ex]{##2}\Ar[inner sep =0.2ex]{ld}{##6}{B}\\
							\&{##3}
							\ifx#6l{\Arb{Rightarrow,shift right=#7em}{#8}{A}{B}{#9}}\else{\ifx#6o{\Arbs{Rightarrow,shift right=#7em}{#8}{B}{A}{#9}}\else{\ifx#6i{\Arbs[inner sep=0.9ex]{iso,shift right=#7em}{#8}{A}{B}{#9}}\else{\ifx#6e{\Arb{equal,shift right=#7em}{#8}{A}{B}{#9}}\else{}\fi}\fi}\fi}\fi
					\end{cd}}%
					\footr}

\begin{document}
					
					\title{Fibrational approach to Grandis exactness for 2-categories}
					
					\author{Elena Caviglia}
					\address{(Elena Caviglia) Department of Mathematical Sciences, Stellenbosch University, South Africa. National Institute for Theoretical and Computational Sciences (NITheCS), Stellenbosch, South Africa.}
					\email{elena.caviglia@outlook.com}
					
					\author{Zurab Janelidze}
					\address{(Zurab Janelidze) Department of Mathematical Sciences, Stellenbosch University, South Africa.
						National Institute for Theoretical and Computational Sciences (NITheCS), Stellenbosch, South Africa.}
					\email{zurab@sun.ac.za}
					
					\author{Luca Mesiti}
					\address{(Luca Mesiti) Department of Mathematical Sciences, Stellenbosch University, South Africa.}
					\email{luca.mesiti@outlook.com}
					
					

					\begin{abstract} 
						In an abelian category, the (bi)fibration of subobjects is isomorphic to the (bi)fibration of quotients. This property captures substantial information about the exactness structure of a category. Indeed, as it was shown by the second author and T.~Weighill, categories equipped with a proper factorization system such that the opfibration of subobjects relative to the factorization system is isomorphic to the fibration of relative quotients are precisely the Grandis exact categories. In this paper we characterize those $(1,1)$-proper factorization systems on a $2$-category in the sense of M.~Dupont and E.~Vitale, for which the weak $2$-opfibration of relative $2$-subobjects is biequivalent to the weak 2-fibration of relative $2$-quotients. This results in a new notion of $2$-dimensional exactness, which we then compare with similar notions in the context of categories enriched in pointed groupoids arising in the work of M.~Dupont and H.~Nakaoka.
					\end{abstract}
					
					\maketitle
					
					\setcounter{tocdepth}{1}
					
					\section{Introduction} 
					
					One of the central notions in homological algebra is that of an exact sequence. In categorical algebra, there is a wide range of axiomatic frameworks where exact sequences can be defined and studied. At the center of these frameworks is the classical notion of an abelian category, which includes categories of modules over a ring, and in particular, the category of $\mathbb{Z}$-modules (i.e., abelian groups), as well as sheaves of modules.
					
					Extension of homological algebra to $2$-dimensional categories is an on-going endeavor that is in its early phase of development. Motivated by homological properties of symmetric categorical groups (i.e., internal abelian pseudo-groups in the category of categories), a notion of $2$-dimensional abelian category was proposed and studied by M.~Dupont \cite{Dup08}. There is also a related work on $2$-dimensional cohomology theory by H.~Nakaoka \cite{Nak08, Nak10}. This paper aims at expanding 2-dimensional homological algebra, with insights from the theories of 2-dimensional fibrations and 2-dimensional factorization systems.
					
					Abelian categories have the following fibrational expression of some of their homological properties: the opfibration of subobjects is isomorphic to the fibration of quotients. In fact, as shown in \cite{JanWei16}, this property is characteristic to Grandis exact categories, which provide an abstract context for capturing good behaviour of exact sequences in abelian categories. In this paper we develop a 2-dimensional analogue of this fibrational approach to exactness, which allows us to introduce a notion of a Grandis exact category in dimension 2. In a forthcoming joint work with \"Ulo Reimaa \cite{JanRei24}, we prove that the 2-category of
					abelian categories (with suitably chosen morphisms) is an example of our notion.
					
					In the pointed case, we obtain a 2-dimensional notion of Puppe exactness that we then compare with similar notions developed by M.~Dupont \cite{Dup08} and H.~Nakaoka \cite{Nak08,Nak10} in the context of categories enriched over pointed groupoids. To do so, we introduce a slightly weaker notion of 2-dimensional Puppe exactness that only requires a weakly closed 2-ideal instead of a closed one.
					This notion generalizes most of the $2$-dimensional notions of exactness developed in \cite{Dup08,Nak08,Nak10}. Moreover, we show that Puppe exact 2-categories in this weaker sense still enjoy
					a $2$-dimensional version of the equivalence between subobjects and quotients; hence the same is true for most of the contexts and all the examples considered in \cite{Dup08,Nak08,Nak10}, and in particular for the $2$-category of symmetric categorical groups and that of 2-dimensional vector spaces in the sense of Baez and Crans \cite{BaezCrans}. We also note that our theory expands the class of examples to include all $1$-dimensional abelian categories.

					\section{2-ideals, 2-kernels, 2-cokernels}\label{sectiontwoideals}
					
					A key tool for our investigation will be a notion of a 2-dimensional ideal (of null morphisms and null 2-cells). This will be a 2-dimensional analogue of the well known concept of an ideal of null morphisms (see \cite{Ehr64}). In this section, we propose a notion of 2-ideal, via a profunctor approach. We then explicitly characterize such a notion in terms of a class of null morphisms and a class of null 2-cells that satisfy 2-dimensional analogues of the familiar condition of closure under composition with morphisms in the category.
					
					Moreover, we introduce 2-dimensional kernels and cokernels, with respect to a 2-ideal. In the particular case of strictly described categories enriched over pointed groupoids, our notions recover the known notions of 2-kernel and 2-cokernel (see \cite{Dup08,Nak08}).
					
					\begin{remark}
						We notice that the notion of ideal (of null morphisms) in a category can be abstractly captured via a profunctor approach. Indeed, an ideal of null morphisms $\N$ in a 1-category $\C$ is the same as a subprofunctor $N$ of the profunctor $\Hom{-}{-}$, or more precisely a pair $(N,\nu)$ where $N\: \C\op \x \C \to \Set$ is a profunctor and $\nu$ is a natural transformation
						\begin{cd}
							{\C\op \x \C} \arrow[r, "N"{name=R}, bend left=30]  \arrow[r, "\Hom{-}{-}"'{name=S}, bend right=30] \arrow[from=R, to =S, Rightarrow, shorten <= 1.5ex, shorten >= 1.5ex, "\nu"] \& {\Set}
						\end{cd}
						with injective components. Notice that the functoriality conditions together with the naturality of $\nu$ give the usual conditions of closure under composition. This is similar to the common practice of viewing sieves as subfunctors of representable presheaves.
					\end{remark}
					
					This then motivates the following definition.
					
					\begin{definition}\label{deftwoideal}
						A \dfn{2-ideal} $\N$ (of null morphisms and null 2-cells) in a 2-category $\L$ is a pair $(N, \nu)$ where $N\:\L \op \x \L \to \Cat$ is a normal pseudofunctor, i.e.\ a 2-dimensional profunctor, and $\nu$ is an injective on objects and faithful pseudo natural transformation
						\begin{cd}
							{\L\op \x \L} \arrow[r, "N"{name=R}, bend left=30]  \arrow[r, "\Hom{-}{-}"'{name=S}, bend right=30] \arrow[from=R, to =S, Rightarrow, shorten <= 1.5ex, shorten >= 1.5ex, "\nu"] \& {\Cat.}
						\end{cd}
					\end{definition}
					
					\begin{remark}
						In the notations of \defx\ref{deftwoideal}, for every $A,B\in \L$, the category $N(A,B)$ represents the category of null morphisms from $A$ to $B$ and null 2-cells between them. Pseudonaturality of $\nu$ captures the following property: given a null morphism $n\:A\anull{} B$ and morphisms $a\:A'\to A$, $b\:B\to B'$ in $\L$, the composite $b\c n\c a$ is isomorphic to a null morphism. The composite itself may not be null, but there exists some replacement, isomorphic to it, which is null. And this replacement is given by $N(a,b)(n)$. This is similar to the notion of bisieve, introduced by Street in \cite{Str82}. A difference from \cite{Str82} is that we do not take all 2-cells to be null. Our $\nu$ is indeed not full in general. This is essential to capture the expected properties of the 2-ideal $(0)$ in the 2-pointed case, see \exax\ref{zero2ideal}. Indeed, we want $N(A,B)$ to be equivalent to the singleton category in the 2-pointed case, exactly as $N(A,B)$ is isomorphic to the singleton set in the $1$-pointed case.
					\end{remark} 
					
					We can now characterize 2-ideals in terms of a class of null morphisms and a class of null 2-cells
					
					\begin{theorem}\label{theorchar2ideal}
						Let $\L$ be a 2-category. A $2$-ideal $\N$ in $\L$ (of null morphisms and null $2$-cells) can be equivalently given as
						\begin{itemize}
							\item[-] a class of morphisms, called null morphisms, and a class of 2-cells between them, called null 2-cells such that the identity 2-cell on a null morphism is a null 2-cell and null 2-cells are closed under vertical composition;
							\item[-] for every null morphism $A \anull{n} B$ and every pair $(a,b)$ with $A'\aar{a} A$ and $B \aar{b} B'$ morphisms in $\L$, a null morphism $A' \anull{\widetilde{b \c n\c a}}B'$ and an isomorphic 2-cell 
							\begin{cd}
								{A'} \arrow[r,"a"]  \arrow[rrr,"\widetilde{b \c n \c a}"',bend right=30, null] \& {A} \arrow[r,"n",null] \arrow[r,"\nu_{a,b,n}"',iso, shift right= 2.3ex] \& {B} \arrow[r,"b"] \& {B'}
							\end{cd}
						\end{itemize}
						such that
						\begin{itemize}
							\item[$(1)$] for every null morphism $A \anull{n} B$ we have $\widetilde{\id{} \c n \c \id{}} =n$ and $\nu_{\id{},\id{},n}=\id{}$;
							\item[$(2)$] for every null 2-cell $\mu\:n \Rightarrow n'$ between any null morphisms $n,n': A \anull{} B$ and for every pair of morphism $(a,b)$ with $A'\aar{a} A$ and $B \aar{b} B'$ the 2-cell
							\begin{cd}
								{A'} \arrow[r,"a"] \arrow[rrr,"\widetilde{b \c n \c a}", bend left=47, null] \arrow[rrr,"\widetilde{b \c n' \c a}"', bend right=48.5, null]\& {A} \arrow[r,"{\hspace{3.05em}\nu^{-1}_{a,b,n}}"'{description, sep=1ex}, iso, shift left=4.75ex] \arrow[r,"{\hspace{3.35em}\nu_{a,b,n'}}"'{description,sep=2.5ex}, iso, shift right=5ex]\arrow[r,"n"{name=B}, null, bend left] \arrow[r,"n'"'{name=A}, null, bend right] \arrow[from=B, to=A,"\mu"',Rightarrow, shorten <= 1ex, shorten >= 1ex]\& {B} \arrow[r,"b",] \& {B'}
							\end{cd}
							is a null 2-cell;
							\item[$(3)$] for every null morphism $A \anull{n} B$ and every pair of 2-cells $(\alpha,\beta)$ with $\alpha\: a \Rightarrow a' \: A' \to A$ and $\beta\: b\Rightarrow b' \: B \to B'$ the 2-cell
							\begin{cd}
								{A'} \arrow[r,"a"{name=D}, bend left] \arrow[r,"a'"'{name=C}, bend right] \arrow[rrr,"\widetilde{b \c n \c a}", bend left=43, null] \arrow[rrr,"\widetilde{b' \c n \c a'}"', bend right=44.5, null]\& {A} \arrow[r,"{\hspace{3.35em}\nu^{-1}_{a,b,n}}"'{description, sep=1ex}, iso, shift left=4ex] \arrow[r,"{\hspace{3.35em}\nu_{a',b',n}}"'{description,sep=2.5ex}, iso, shift right=4ex]\arrow[r,"n"{name=B}, null]  \arrow[from=D, to=C,"\alpha"',Rightarrow, shorten <= 1ex, shorten >= 1ex]\& {B} \arrow[r,"b"{name=E},bend left] \arrow[r,"b'"'{name=F},bend right] \arrow[from=E, to=F,"\beta"',Rightarrow, shorten <= 1ex, shorten >= 1ex]\& {B'}
							\end{cd}
							ia a null 2-cell;
							\item[$(4)$] for every null morphism $A \anull{n} B$ and every two pairs $(a,b)$, $(a',b')$ with $A'\aar{a} A$, $A''\aar{a'} A'$, $B \aar{b} B'$, $B'\aar{b'} B''$ morphisms in $\L$ the 2-cell
							\begin{cd}
								{A''} \arrow[r,"a'"] \arrow[rrrrr, bend left=41.5,null] \arrow[rrrrr, bend right=23,null]  \& {A'} \arrow[rrr, bend left=32,null] \arrow[r,"a"'] \& {A} \arrow[r, "n"'{inner sep=0.8ex}, null] \arrow[r,"{\hspace{4.65em}\nu^{-1}_{a',b',\widetilde{b \c n \c a}}}"'{description,sep=2.5ex}, iso, shift left=7.65ex] \arrow[r,"{\hspace{3.35em}\nu^{-1}_{a,b,n}}"'{description,sep=2.5ex}, iso, shift left=2.65ex] \arrow[r,"{\hspace{4.65em}\nu_{a\c a',b'\c b,n}}"'{description,sep=2.5ex}, iso, shift right=3.5ex] \& {B} \arrow[r, "b"'] \& {B'} \arrow[r, "b'"] \& {B''}
							\end{cd}
							is an invertible null 2-cell.
						\end{itemize}
					\end{theorem}
					\begin{proof}
						Consider the $2$-ideal $\N=(N,\nu)$ in $\L$. For every $(A,B)\in \L \x \L\op$, the normal pseudofunctor $N$ gives a subcategory $N(A,B)$ of the hom-category $\Hom{A}{B}$ consisting of null morphisms and null 2-cells between them. This means that the the identity 2-cell on a null morphism is a null 2-cell and null 2-cells are closed under vertical composition. Moreover, given morphisms $A'\aar{a}A$ and $B \aar{b} B'$ in $\L$, the pseudonatural transformation $\mu$ gives an invertible 2-cell
						\begin{cd}
							{N(A,B)} \arrow[r,iso, "\nu_{a,b}"', shift right= 5ex] \arrow[d,"{N(A,B)}"'] \arrow[r, hook] \& {\Hom{A}{B}} \arrow[d,"{b \c - \c a}"] \\
							{N(A',B')} \arrow[r, hook] \& {\Hom{A'}{B'}.}
						\end{cd}
						It is then straightforward to check that the naturality of $\nu_{a,b}$ corresponds to (2). Moreover, it is straightforward to prove that the pseudonaturality of $\nu$ is equivalent to conditions (1), (3) and (4). In particular, condition (1) and (4) are given by the fact that the components of $\nu$ respect identities and composition while condition (3) is the compatibility with 2-cells of the components of $\nu$. 
						
						Conversely, given a class of morphisms, a class of 2-cells and isomorphic null 2-cells satisfying the axioms from (1) to (4), for every $A,B\in \L$ we define $N(A,B)$ to be the category consisting of the the morphisms from $A$ to $B$ that are in the chosen class of morphisms and the 2-cells between them that are in the chosen class of 2-cells. Such $N(A,B)$ is a subcategory of $\Hom{A}{B}$, so we have an injective on objects and faithful functor $N(A,B)\to\Hom{A}{B}$. Given morphisms $A'\aar{a}A$ and $B \aar{b} B'$ in $\L$ and a null morphism $A \anull{n} B$, we define $N(a,b)(n)$ as $\widetilde{b \c n \c a}$ and $(\nu_{a,b})_n$ as $\nu_{a,b,n}$. Moreover, given a 2-cell $\mu\: n \rightarrow n'$, we define $N(a,b)(\mu)$ as the 2-cell of condition (2). And given a null morphism $A \anull{n} B$ and a pair of 2-cells $(\alpha,\beta)$ with $\alpha\: a \Rightarrow a' \: A' \to A$ and $\beta\: b\Rightarrow b' \: B \to B'$, we define the 2-cell $N(\alpha,\beta)_n$ as the 2-cell of condition (3). Condition (4) then gives the isomorphisms that regulate the pseudofunctoriality of $N$. Finally, it is straightforward to prove that $\nu$ is a pseudonatural transformation. So $(N,\nu)$ is a 2-ideal in the sense of \defx\ref{deftwoideal}.
					\end{proof}
					
					We now aim at describing the 2-ideal canonically associated to a 2-pointed 2-category.
					
					\begin{definition}
						We call a 2-category $\L$ \dfn{2-pointed} if $\L$ has a biterminal object which is also a biinitial object. We will call such object \dfn{bizero object} and denote it as $0$.
					\end{definition}
					
					\begin{remark}
						A bizero object $0$ in a 2-category $\L$ is such that, for every $A\in \L$ both the hom-categories
						$\HomC{\L}{A}{0}$ and $\HomC{\L}{0}{A}$ are equivalent to the terminal category. This means that, for every $A\in \L$, there exist morphisms $t\:A\to 0$ and $i\:0\to A$ (which need not be unique), and given any two morphisms $t,t'\:A\to 0$ (or any two morphisms $i,i'\:\0\to A$) there exists a unique isomorphic 2-cell between them.
					\end{remark}
					
					\begin{example}\label{zero2ideal}
						Let $\L$ be a 2-pointed 2-category, with a fixed choice $0$ of a representative for the bizero object. Then $\L$ has a canonical 2-ideal $(0)$ associated to it, given as follows. Null morphisms are precisely the ones that factor through $0$ and null 2-cells are precisely the ones induced by the 2-dimensional universal property of the bizero object as shown in the following diagram
						\begin{eqD}{nulltwocell}
							\begin{cd}*
								{} \& {0} \arrow[dd, equal] \arrow[rd,"i", bend left=20]\& {}\\[-5ex]
								{A}  \arrow[r,iso]\arrow[ru,"t", bend left=20] \arrow[rd,"t'"', bend right=20] \& {} \arrow[r,iso] \& {B.}\\[-5ex]
								{} \& {0} \arrow[ru,"i'"', bend right=20] \& {}
							\end{cd}
						\end{eqD}
						Given a null morphism $n=A\aar{t} 0 \aar{i} B$ and morphisms $A'\aar{a} A$ and $B \aar{b} B'$, we notice that the composite $A' \aar{a} A\aar{t} 0 \aar{i} B \aar{b} B'$ is null. So we define  $\widetilde{b \c n \c a}$ as such composite and $\nu_{a,b,n}$ as the identity 2-cell. It is then straightforward to check that these data satisfy the required axioms (see \thex\ref{theorchar2ideal}) using the universal properties of the bizero object.
						
						Notice that, for every $A,B\in \L$, we have that $(0)(A,B)$ is equivalent to the terminal category. Indeed, the functor from the terminal category to $(0)(A,B)$ that picks a fixed chosen morphism $A \aar{t} 0 \aar{i} B$ is an equivalence of categories since given any other null morphism $A \aar{t'} 0 \aar{i'} B$ the 2-cell of diagram (\ref{nulltwocell}) is an isomorphic null 2-cell.
						
						This is a 2-dimensional analogue of the fact that $N(A,B)$ is isomorphic to the singleton set in the 1-pointed case (as there is a unique null morphism $A\to 0\to B$ for every $A$ and $B$ in the 1-pointed case).
					\end{example}
					
					\begin{remark}\label{remadditionalaxiomsideal}
						There are several additional axioms that a 2-ideal of null morphisms and null 2-cells might satisfy. For example, one could consider to ask the components of $\nu$ to be full, conservative, isofibrations, or discrete opfibrations.
					\end{remark}
					
					\begin{remark}\label{remtwomonos}
						The additional axioms explored in \remx\ref{remadditionalaxiomsideal} correspond to possible 2-dimensional notions of monomorphism. Indeed, the components of the pseudonatural transformation $\nu$ of a 2-ideal should be 2-dimensional monomorphisms in $\Cat$, exactly as they are monomorphisms in $\Set$ (i.e.\ injections) in dimension 1. From the point of view of elementary topos theory, discrete opfibrations give a nice notion of subobjects in dimension 2, see \cite{Web07} and \cite{Mes25}. In many other examples, though, discrete opfibrations seem to be a bit too strong.
						
						Following Dupont and Vitale \cite{DupVit03}, we will consider the basic notion of monomorphism in dimension 2 given by faithful morphisms. Epimorphisms in dimension 2 will be given by cofaithful morphisms. In the terminology of \cite{DupVit03}, these correspond to $(1,1)$-proper. Considering other notions of monomorphisms and epimorphisms in dimension 2, one would get corresponding variations of our results. 
					\end{remark}
					
					We are ready to define 2-dimensional kernels and cokernels relative to a 2-ideal.
					
					\begin{definition}\label{deftwokernel}
						Let $\L$ be a 2-category and $\N=(N,\nu)$ be a 2-ideal of null morphisms and null 2-cells in $\L$. An \dfn{$\N$-2-kernel (or just 2-kernel)} of a morphism $f\:A\to B$ in $\L$ is a morphism $K \aar{k} A$ together with an isomorphic 2-cell
						\begin{cd}
							K\ar[rr,null,bend right=35,"n"'{inner sep=0.8ex}]\ar[rr,iso,shift right=2ex,"\alpha"'{pos=0.6}]\ar[r,"k"] \& A \ar[r,"f"] \& B
						\end{cd}
						with $n$ a null morphism, that is universal in the following sense:
						\begin{itemize}
							\item[$(1)$] for every morphism $Z \aar{z} A$ such that $f \c z$ is isomorphic to a null morphism via an isomorphism $\beta$, there exist a morphism $Z \aar{u} K$ and an isomorphic 2-cell $\gamma\: z \Rightarrow k \c u$ such that the pasting
							\begin{cd}
								{Z} \arrow[rrrd,""{name=C}, null, bend right=69] \arrow[rd,"{\exists u}"'{name=A}, dashed] \arrow[rrd,"z"{name=B}, bend left=20] \arrow[rrrd,"", null, bend left] \&[-2ex] {} \&[-2ex] {} \&[-2ex] {}\\[-4ex]
								{} \& {K} \arrow[r,"{\exists \gamma}"'{pos=0.65}, iso, shift left= 3ex, start anchor={[xshift=-8ex]}] 
								\arrow[r,"{\nu}"'{pos=0.65}, iso, shift right= 2.8ex, start anchor={[xshift=-8ex]}] 
								\arrow[r,"k"] \arrow[rr,""{name=D}, bend right=42, null] \& {A} \arrow[r,"{\beta}"'{pos=0.65}, iso, shift left= 3.3ex, start anchor={[xshift=-8ex]}] \arrow[r,"{\alpha}"'{pos=0.65}, iso, shift right= 2ex, start anchor={[xshift=-8ex]}]  \arrow[r,"f"] \& {B}
							\end{cd}
							is an invertible null 2-cell;
							\item[$(2)$] for every morphisms $u,v\: Z \to K$ and every 2-cell $\lambda\: k\c u \Rightarrow k \c v$ such that the 2-cell $\Xi^\lambda$ given by the following pasting
							\begin{cd}
								{Z} \arrow[rrr,"", null, bend right=60] \arrow[rrr,"", null, bend left=50]\arrow[r,"u"] \arrow[rd,"v"'] \& {K} \arrow[r,"{\nu}"'{pos=0.65}, iso, shift left= 4.8ex, start anchor={[xshift=-8ex]}] \arrow[d,"\lambda", Rightarrow, shorten <= -0.2ex, shorten >= -0.2ex] \arrow[r,"k"] \arrow[rr,"", null, bend left=35] \arrow[rr,"{\alpha^{-1}}"'{pos=0.65, inner sep=0.1ex}, iso, shift left= 2.5ex] \arrow[rr,"{\alpha}"'{pos=0.65, inner sep=0.1ex}, iso, shift right= 2.5ex]\& {A} \arrow[r,"f"] \& {B}\\[-4.5ex]
								{} \& {K} \arrow[r,"{\nu}"'{pos=0.75}, iso, shift right= 2.1ex, start anchor={[xshift=-0.2ex]}] \arrow[ru,"k"'] \arrow[rru,"", null, bend right=20] \& {} \& {}
							\end{cd}
							is a null 2-cell, there exists a unique 2-cell $\mu\: u \Rightarrow v$ such that $k \star \mu = \lambda$.
						\end{itemize}

						Dually, we can define \predfn{$\N$-2-cokernels}. Explicitly, the 1-dimensional universal property is exhibited by the following diagram
						\begin{cd}
							{} \&[-2ex] {} \&[-2ex] {} \&[-2ex] {Z}\\[-4ex]
							{A} \arrow[rr,"{\hspace{2ex}\alpha}"'{pos=0.57, inner sep=0.01ex}, iso, shift right= 1.5ex, start anchor={[xshift=-0.1ex]}] \arrow[rr,"{\hspace{2ex}\beta}"'{pos=0.57, inner sep=0.001ex}, iso, shift left= 4.5ex, start anchor={[xshift=-0.1ex]}] \arrow[rrru,"", null, bend left=42] \arrow[rr,"", null, bend right=30] \arrow[rrru,"", null, bend right=72] \arrow[r,"f"] \& {B} \arrow[r,"{\hspace{2ex}\nu}"'{pos=0.99}, iso, shift right= 3.1ex, start anchor={[xshift=3.5ex]}] \arrow[r,"{\hspace{2ex}\gamma}"'{pos=0.99}, iso, shift left= 2.1ex, start anchor={[xshift=3.2ex]}] \arrow[rru,"z", bend left=20]\arrow[r,"c"] \& {C} \arrow[ru,"{\exists u}"', dashed] \& {}
						\end{cd}
						and analogously for the 2-dimensional universal property. 
					\end{definition}
					
					\begin{proposition} \label{propkerarefaith}
						$\N$-2-kernels are faithful morphisms, i.e.\ 2-dimensional monomorphisms (see \remx\ref{remtwomonos}). Dually, $\N$-2-cokernels are cofaithful morphisms, i.e.\ 2-dimensional epimorphisms.
					\end{proposition}
					\begin{proof}
						Straightforward using the 2-dimensional universal property of $\N$-2-kernels.
					\end{proof}
					
					\begin{lemma}
						In the notations of \defx\ref{deftwokernel}, if $\lambda$ is an isomorphic 2-cell and $\Xi^\lambda$ is an invertible null 2-cell, then the 2-cell $\mu$ induced by the 2-dimensional universal property of the 2-kernel is an isomorphic 2-cell.
						
						Dually for 2-cokernels. 
					\end{lemma}
					\begin{proof}
						Applying the 2-dimensional universal property of the 2-kernel to the 2-cell $\lambda^{-1}\: k \c v \Rightarrow k \c u$, we construct a 2-cell $\mu'\: v \Rightarrow u$ such that $k \star \mu'=\lambda^{-1}$. This implies $k \star (\mu' \c \mu)=\lambda^{-1}$ and analogously $k \star (\mu \c \mu')=\lambda^{-1}$. Since $k$ is faithful by \prox\ref{propkerarefaith}, we conclude that $\mu'$ is the inverse of $\mu$.
					\end{proof}
					
					\begin{proposition}
						$\N$-2-kernels and $\N$-2-cokernels are uniquely determined up to equivalence. More precisely, given two $\N$-kernels $K\aar{k} A$ and $K'\aar{k'} A$ of a morphism $f\:A\to B$ in $\L$, there exist an equivalence $j\:K\to K'$ in $\L$ and an isomorphic 2-cell
						\tr[3][5][0][0][i][-0.6][\gamma]{K}{K'}{A}{j}{k}{k'}
						such that $(j,\gamma)$ is an equivalence in the pseudo slice 2-category over $A$ (i.e.\ $j$ is an equivalence over $A$, taking into account the isomorphic 2-cell $\gamma$).
					\end{proposition}
					\begin{proof}
						Straightforward using the universal properties of $\N$-2-kernels and $\N$-2-cokernels.
					\end{proof}
					
					We show that in the 2-pointed case, 2-kernels are bilimits. More precisely, they are biisoinserters. This is a 2-dimensional analogue of the fact that in the 1-pointed case kernels are equalizers (of the morphism with the zero morphism).
					
					We first recall the definition of the biisoinserter.
					
					\begin{definition}
						Let $\L$ be a 2-category and let $f,g\:A\to B$ be two parallel morphisms in $\L$. The \dfn{biisoinserter of $f$ and $g$} is, if it exists, an object $L\in \L$ together with a morphism $\ell\:L\to A$ and an isomorphic 2-cell
						\begin{cd}[1][5]
							\& A \ar[rd,"{f}"]\& \\
							L \ar[ru,"{\ell}"]\ar[rd,"{\ell}"'] \ar[rr,iso,shift right=0.4ex,"{\lambda}"{inner sep=1.15ex}]\&\& B \\
							\& A \ar[ru,"{g}"'] \&
						\end{cd}
						in $\L$, which is universal in the following bicategorical sense:
						\begin{enum}
							\item for every morphism $m\:M\to A$ and isomorphic 2-cell $\mu\:f\c m\iso g\c m$, there exist a morphism $v\:M\to L$ and an isomorphic 2-cell $\delta\:m\iso \ell\c v$ such that
							\begin{eqD*}
								\begin{cd}*[1][5]
									\& A \ar[rd,"{f}"]\& \\
									M \ar[ru,"{m}"]\ar[rd,"{m}"'] \ar[rr,iso,shift right=0.4ex,"{\mu}"{inner sep=1.15ex}]\&\& B \\
									\& A \ar[ru,"{g}"'] \&
								\end{cd}\quad = \h[-0.5]\quad
								\begin{cd}*[1][5]
									\&\& A \ar[rd,"{f}"]\& \\
									M \ar[rr,iso,shift left=2.6ex,"{\delta}"{inner sep=1.15ex}]\ar[rr,iso,shift right=2.6ex,"{\delta^{-1}}"'{inner sep=0.92ex}]\ar[r,"{v}"]\ar[rru,bend left,"{m}"]\ar[rrd,bend right,"{m}"'] \& L \ar[ru,"{\ell}"]\ar[rd,"{\ell}"'] \ar[rr,iso,shift right=0.4ex,"{\lambda}"{inner sep=1.15ex}]\&\hphantom{.}\& B \\
									\&\& A \ar[ru,"{g}"'] \&
								\end{cd}
							\end{eqD*}
							\item for every pair of morphisms $v,w\:M\to L$ and 2-cell $\tau\:\ell\c v\aR{}\ell\c w\:M\to A$ in $\L$ such that $\lambda\ast w\c f\ast \tau=g\ast \tau\c \lambda\ast v$, there exists a unique 2-cell $\sigma\:v\aR{}w$ such that $\ell\ast \sigma=\tau$.    
						\end{enum}
					\end{definition}
					
					\begin{proposition}
						Let $\L$ be a 2-pointed 2-category, with a fixed choice $0$ of a representative for the bizero object. Consider the 2-ideal $(0)$ canonically associated to $\L$. Then the 2-kernel of a morphism $f\:A\to B$ is given by the biisoinserter of $f\:A\to B$ and a chosen null morphism $A\to 0\to B$ (one always exists because $0$ is a bizero object).
						
						As a consequence, $2$-kernels are faithful and conservative in the 2-pointed case.
						
						Dually for 2-cokernels.
					\end{proposition}
					\begin{proof}
						We prove that the structure 2-cell of the biisoinserter of $f$ and the chosen $A\aar{t}0\aar{i}B$ gives the structure isomorphic 2-cell 
						\begin{cd}[0.5][5]
							\& A \ar[rrd,bend left=14,"{f}"]\&\& \\
							L \ar[ru,"{\ell}"]\ar[rd,"{\ell}"'] \ar[rrr,iso,shift right=0.4ex,"{\lambda}"{inner sep=1.15ex}]\&\&\& B \\
							\& A \ar[r,"{t}"'] \& 0 \ar[ru,"{i}"']
						\end{cd}
						of a 2-kernel of $f$. Notice that indeed the composite morphism below is null.
						
						Given a morphism $M\aar{m} A$ and an isomorphic 2-cell
						\begin{cd}[1][5]
							\& 0 \ar[rd,"{b}"]\& \\
							M \ar[ru,"{a}"]\ar[rd,"{m}"'] \ar[rr,iso,shift right=0.4ex,"{\beta}"{inner sep=1.15ex}]\&\& B \\
							\& A \ar[ru,"{f}"'] \&
						\end{cd}
						we can paste $\beta^{-1}$ with the unique isomorphic 2-cells $a\iso t\c m$ and $b \iso i$ to obtain an isomorphic 2-cell $f\c m\iso i\c t\c m$, and thus induce $v\:M\to L$ and an isomorphic 2-cell $\delta:m\iso \ell\c v$ by the universal property of the biisoinserter. It is straightforward to show that $v$ and $\delta$ satisfy the required axiom for the 1-dimensional universal property of the 2-kernel.
						
						Consider then $v,w\:M\to K$ and a 2-cell $\tau\:\ell\c v\aR{}\ell\c w$ such that $f\ast \tau$, pasted with $\lambda$ and $\lambda^{-1}$, gives a null 2-cell, which means that it factors through $i$. The latter assumption guarantees that we can apply the 2-dimensional universal property of the biisoinserter to the 2-cell $\tau$. And we thus obtain that there exists a unique 2-cell $\sigma\:v\aR{}w$ such that $\ell\ast \sigma=\tau$.
					\end{proof}
					
					\begin{remark}
						Biisoinserters are a 2-dimensional analogue of equalizers. Indeed, they equalize parallel morphisms up to an isomorphic 2-cell, in a 2-universal way.
					\end{remark}
					
					\begin{remark}\label{remcomparisonkernels}
						It is straightforward to see that in the particular case of (strictly described) categories enriched over pointed groupoids, the 2-pointed version of our notions recovers the known notions of 2-kernel and 2-cokernel in the literature (see \cite{Dup08,Nak08}). Our notions are much more general, as they do not require to have a bizero object nor all 2-cells to be isomorphisms.
					\end{remark}
					
					The notion of Grandis exact category involves closed ideals of null morphisms. We thus need a 2-dimensional analogue of closedness for an ideal.
					
					Recall that the following well known result holds in dimension 1. 
					
					\begin{proposition}\label{propcharclosed}
						Let $\C$ be a category and $\N$ be an ideal of null morphisms in $\C$. Assume that $\C$ has all kernels and cokernels. The following conditions are equivalent:
						\begin{enumT}
							\item $\N$ is closed, i.e.\ every null morphism factors through an object whose identity is a null morphism (called null object);
							\item All kernels reflect null morphisms, i.e.\ if $k\c f$ is a null morphism with $k$ a kernel then $f$ is a null morphism.
							\item All cokernels coreflect null morphisms, i.e.\ if $f\c c$ is a null morphism with $c$ a cokernel then $f$ is a null morphism.
						\end{enumT}
					\end{proposition}
					
					Conditions $(ii)$ and $(iii)$ of \prox\ref{propcharclosed} are easier and more straightforward to generalize to dimension 2. 
					
					\begin{definition}\label{defreflnull}
						Let $\L$ be a 2-category and $\N=(N,\nu)$ be a 2-ideal of null morphisms and null 2-cells in $\L$. We say that a morphism $K \aar{k} A$ \dfn{reflects null morphisms} if whenever a morphism $D \aar{s} K$ is such that $D \aar{s} K \aar{k} A$ is isomorphic to a null morphism $D \anull{n} A$ via a 2-cell $\delta\: k \c s \Rightarrow n$, there exist a null morphism $D \anull{\widehat{\psi}} K$ and an isomorphic 2-cell $\psi\: s \Rightarrow \hat{\psi}$ such that $\delta$ coincides with
						\begin{cd}[7][7]
							{D} \arrow[r, "s", bend left=30] \arrow[r, "{\widehat{\psi}}"'{pos=0.65, inner sep=0.2ex}, bend right=30, null]  \arrow[rr, "{m}"'{pos=0.58, inner sep=0.15ex}, bend right=42, null] \arrow[r, "\psi"'{pos=0.81, inner sep=0.00001ex}, iso]  \& {K} \arrow[r,"{\hspace{2ex}\nu}"'{pos=0.5}, iso, shift right= 2.1ex, start anchor={[xshift=-3.8ex]}]   \arrow[r, "k"]  \& {A}
						\end{cd}
						up to an invertible null 2-cell $\xi\: m \Rightarrow n$. Equivalently, the 2-cell 
						\begin{cd}[7][7]
							{D} \arrow[rr,"n"{pos=0.58,inner sep=0.15ex}, null, bend left=45] \arrow[rr, iso, "\delta^{-1}"'{pos=0.65, inner sep=0.0001ex}, shift left=3ex] \arrow[r, "s", bend left=30] \arrow[r, "{\widehat{\psi}}"'{pos=0.65, inner sep=0.2ex}, bend right=30, null]  \arrow[rr, "{m}"'{pos=0.58, inner sep=0.15ex}, bend right=46, null]  \arrow[r, "\psi"'{pos=0.81, inner sep=0.00001ex}, iso]  \& {K} \arrow[r,"{\hspace{2ex}\nu}"'{pos=0.5}, iso, shift right= 2.1ex, start anchor={[xshift=-3.8ex]}]   \arrow[r, "k"]  \& {A}
						\end{cd}
						is an invertible null 2-cell.

						We say that $K \aar{k} A$ \dfn{reflects null 2-cells} if 
						whenever a 2-cell $\mu\:s\aR{}s'\:D\to K$ between null morphisms $s$ and $s'$ is such that the 2-cell 
						\begin{cd}
							{D} \arrow[rr, shift left=3.5 ex, iso, "\nu"'{inner sep=1ex}] \arrow[rr, shift right=2.5 ex, iso, "\nu"'{inner sep=1ex}] \arrow[rr, bend left=45, null] \arrow[rr, bend right=45, null] \arrow[r, "s"{name=A}, bend left, null] \arrow[r, "s'"'{name=B, inner sep=0.3 ex}, bend right, null] \arrow[from=A, to=B, Rightarrow, shorten >= 1 ex,shorten <= 1 ex, "\mu"] \& {K} \arrow[r, "k"] \& {A}
						\end{cd}
						is a null 2-cell, we have that the 2-cell $\mu$ is a null 2-cell.
						
						Dually for morphisms coreflecting null morphisms and null 2-cells.
					\end{definition}
					
					\begin{definition}
						Let $\L$ be a 2-category and $\N=(N,\nu)$ be a 2-ideal of null morphisms and null 2-cells in $\L$. Assume that $\L$ has all $\N$-2-kernels and $\N$-2-cokernels. We call $\N$ a \dfn{closed 2-ideal} if all $\N$-2-kernels reflect null morphisms and null 2-cells and all $\N$-2-cokernels coreflect null morphisms and null 2-cells. 
					\end{definition}
					
					We show that a generalization to dimension 2 of condition $(i)$ of \prox\ref{propcharclosed} is equivalent to having a weaker version of a closed 2-ideal.
					
					\begin{definition}\label{defweaklyreflect}
						Let $\L$ be a 2-category and $\N=(N,\nu)$ be a 2-ideal of null morphisms and null 2-cells in $\L$. We say that an $\N$-2-kernel $K \aar{k} A$, say of $f\:A\to B$ with structure 2-cell $\alpha$, \dfn{weakly reflects null morphisms} if whenever a morphism $D \aar{s} K$ is such that $D \aar{s} K \aar{k} A$ is isomorphic to a null morphism $D \anull{n} A$ via a 2-cell $\delta\: k \c s \Rightarrow n$ such that
						\begin{cd}[5][5]
							D \& K \& A \& B
							\arrow["s"', from=1-1, to=1-2]
							\arrow["n"{pos=0.58,inner sep=0.15ex}, bend left=40,null, from=1-1, to=1-3]
							\arrow[iso,"{\delta^{-1}}"{pos=0.4,inner sep=0.51ex}, shift left=3.8, from=1-1, to=1-3]
							\arrow[iso,"\nu"'{pos=0.58,inner sep=0.75ex}, shift right=5.6, from=1-1, to=1-3]
							\arrow[bend right=41.5, null, from=1-1, to=1-4]
							\arrow[bend left=41.5, null, from=1-1, to=1-4]
							\arrow["k"', from=1-2, to=1-3]
							\arrow[bend right=39,null, from=1-2, to=1-4]
							\arrow[iso,"\alpha"'{pos=0.57,inner sep=0.7ex}, shift right=3.8, from=1-2, to=1-4]
							\arrow[iso,"\nu"{pos=0.58,inner sep=0.75ex}, shift left=5.6, from=1-2, to=1-4]
							\arrow["f"', from=1-3, to=1-4]
						\end{cd}
						is an invertible null 2-cell, there exist a null morphism $D \anull{\widehat{\psi}} K$ and an isomorphic 2-cell $\psi\: s \Rightarrow \hat{\psi}$ such that the 2-cell 
						\begin{cd}[7][7]
							{D} \arrow[rr,"n"{pos=0.58, inner sep=0.15ex}, null, bend left=45] \arrow[rr, iso, "\delta^{-1}"'{pos=0.65, inner sep=0.0001ex}, shift left=3ex] \arrow[r, "s", bend left=30] \arrow[r, "{\widehat{\psi}}"'{pos=0.65, inner sep=0.2ex}, bend right=30, null]  \arrow[rr, "m"'{pos=0.58, inner sep=0.15ex}, bend right=46, null]  \arrow[r, "\psi"'{pos=0.81, inner sep=0.00001ex}, iso]  \& {K} \arrow[r,"{\hspace{2ex}\nu}"'{pos=0.5}, iso, shift right= 2.1ex, start anchor={[xshift=-3.8ex]}]   \arrow[r, "k"]  \& {A}
						\end{cd}
						is an invertible null 2-cell.
						
						Dually for $\N$-2-cokernels weakly coreflecting null morphisms.
					\end{definition}
					
					\begin{remark}
						Notice that \defx\ref{defweaklyreflect} is a weaker version of \defx\ref{defreflnull}, containing the extra assumption that $\delta$ is compatible with the way in which $k$ is a 2-kernel.
					\end{remark}
					
					\begin{definition}
						Let $\L$ be a 2-category and $\N$ be a 2-ideal of null morphisms and null 2-cells in $\L$. An object $Z$ of $\L$ is called a \dfn{null object} if the identity of $Z$ is isomorphic to a null morphism $\widehat{\xi^Z}\:Z\anull{} Z$.
						\begin{cd}[7][9]
							Z \arrow[r,equal,bend left] \arrow[r,bend right,null,"{\widehat{\xi^Z}}"'{inner sep=0.78ex}] \arrow[r,iso,shift right=0.85ex,"{\exists \xi^Z}"{inner sep=0.78ex}]\& Z
						\end{cd}
					\end{definition}
					
					\begin{proposition}\label{charactweaklyreflect}
						Let $\L$ be a 2-category and $\N=(N,\nu)$ be a 2-ideal of null morphisms and null 2-cells in $\L$. Assume that $\L$ has all $\N$-2-kernels. The following conditions are equivalent:
						\begin{enumT}
							\item All $\N$-2-kernels weakly reflect null morphisms;
							\item Every morphism $h\:A\to B$ isomorphic to a null morphism $n$, via an isomorphic 2-cell $\rho$, factors up to an isomorphic 2-cell $\gamma$ through a null object $Z$, making
							\begin{cd}[5][5]
								A \&\&\&\& B \\[-2ex]
								\& Z \&\& Z \& \hphantom{.}
								\arrow["n", null,bend left=20, from=1-1, to=1-5]
								\arrow[iso,"{\rho^{-1}}"'{pos=0.59, inner sep=0.02ex}, shift left=3.3, from=1-1, to=1-5]
								\arrow["h"', from=1-1, to=1-5]
								\arrow[null,bend right=90, from=1-1, to=1-5]
								\arrow["x"', from=1-1, to=2-2]
								\arrow[equal,bend left=34, from=2-2, to=2-4]
								\arrow[null,"{\widehat{\xi^{Z}}}"'{pos=0.45,inner sep=0.4ex}, bend right=15,from=2-2, to=2-4]
								\arrow[iso,"{\xi^Z}"{pos=0.66,inner sep=0.68}, shift left=0.6, from=2-2, to=2-4]
								\arrow[iso,"{\nu}"{pos=0.585,inner sep=0.68}, shift right=4.4,from=2-2, to=2-5]
								\arrow[iso,"\gamma"{pos=0.62,inner sep=0.68ex}, shift left=6.85, start anchor={[xshift=8ex]},from=2-2, to=2-4]
								\arrow["y"', from=2-4, to=1-5]
							\end{cd}
							into an invertible null 2-cell, where $\xi^Z$ is the isomorphic 2-cell associated to the null object $Z$.
							\item All $\N$-2-cokernels weakly coreflect null morphisms;
						\end{enumT}
					\end{proposition}
					\begin{proof}
						We prove $(i)\aR{}(ii)$. Let $h\:A\to B$ be a morphism in $\L$ isomorphic to a null morphism $n$ via a 2-cell $\rho$. Consider then the $\N$-2-kernel $Z\aar{k}{B}$ of the identity of $B$, and call $\alpha$ its structure 2-cell. Then by the universal property of the $\N$-2-kernel, there exist $u\:A\to Z$ and an isomorphic 2-cell $\gamma\:h\aR{}k\c u$ such that
						\begin{cd}
							{A} \arrow[rrrd,""{name=C}, null, bend right=69] \arrow[rd,"{\exists u}"'{name=A}, dashed] \arrow[rrd,"h"{name=B}, bend left=20] \arrow[rrrd, null, bend left,"n"{pos=0.58,inner sep=0.15ex}] \&[-2ex] {} \&[-2ex] {} \&[-2ex] {}\\[-4ex]
							{} \& {Z} \arrow[r,"{\exists \gamma}"'{pos=0.65}, iso, shift left= 3ex, start anchor={[xshift=-8ex]}] 
							\arrow[r,"{\nu}"'{pos=0.65}, iso, shift right= 2.8ex, start anchor={[xshift=-8ex]}] 
							\arrow[r,"k"] \arrow[rr,""{name=D}, bend right=42, null] \& {B} \arrow[r,"{\rho^{-1}}"'{pos=0.66,inner sep=0.35ex}, iso, shift left= 3.3ex, start anchor={[xshift=-8ex]}] \arrow[r,"{\alpha}"'{pos=0.65}, iso, shift right= 2ex, start anchor={[xshift=-8ex]}]  \arrow[r,equal] \& {B}
						\end{cd}
						is an invertible null 2-cell. We show that $Z$ is a null object. The identity $Z\aeqq{}Z$ is such that $k\c \id{Z}$ is isomorphic to a null morphism, via $\alpha$, making
						\begin{cd}[4.5][4.5]
							Z \& Z \& B \& B
							\arrow[equal, from=1-1, to=1-2]
							\arrow[""{pos=0.58,inner sep=0.15ex}, bend left=41,null, from=1-1, to=1-3]
							\arrow[iso,"{\alpha^{-1}}"{pos=0.36,inner sep=0.37ex}, shift left=3.8, from=1-1, to=1-3]
							\arrow[bend right=41, null, from=1-1, to=1-4]
							\arrow[bend left=41, null, from=1-1, to=1-4]
							\arrow["k"', from=1-2, to=1-3]
							\arrow[bend right=40,null, from=1-2, to=1-4]
							\arrow[iso,"\alpha"'{pos=0.57,inner sep=0.7ex}, shift right=3.8, from=1-2, to=1-4]
							\arrow[equal, from=1-3, to=1-4]
						\end{cd}
						into the identity 2-cell and thus into an invertible null 2-cell. Since $k$ weakly reflects null morphisms, there exist a null morphism $\widehat{\xi^{Z}}\:Z\anull{} Z$ and an isomorphic 2-cell $\xi^Z\:\id{Z}\aR{}\widehat{\xi}$ such that the 2-cell
						\begin{cd}[7][7]
							{Z} \arrow[rr,""{pos=0.58, inner sep=0.15ex}, null, bend left=45] \arrow[rr, iso, "\alpha^{-1}"'{pos=0.65, inner sep=0.0001ex}, shift left=3ex] \arrow[r,equal, bend left=30] \arrow[r, "{\widehat{\xi^Z}}"'{pos=0.65, inner sep=0.2ex}, bend right=30, null]  \arrow[rr, ""'{pos=0.58, inner sep=0.15ex}, bend right=46, null]  \arrow[r, "{\xi^Z}"'{pos=0.81, inner sep=0.00001ex}, iso]  \& {Z} \arrow[r,"{\hspace{2ex}\nu}"'{pos=0.5}, iso, shift right= 2.1ex, start anchor={[xshift=-3.8ex]}]   \arrow[r, "k"]  \& {B}
						\end{cd}
						is an invertible null 2-cell. We conclude that $Z$ is a null object and that $(ii)$ holds.
						
						We now prove $(ii)\aR{}(i)$. Let $f\:A\to B$ be a morphism in $\L$ and $K\aar{k}A$ be the $\N$-2-kernel of $f$, with structure 2-cell $\alpha$. Consider then $D\aar{s}K$ and an isomorphic 2-cell $\delta\:k\c s\aR{} n$ with $n$ a null morphism, such that 
						\begin{cd}[5][5]
							D \& K \& A \& B
							\arrow["s"', from=1-1, to=1-2]
							\arrow["n"{pos=0.58,inner sep=0.15ex}, bend left=40,null, from=1-1, to=1-3]
							\arrow[iso,"{\delta^{-1}}"{pos=0.4,inner sep=0.51ex}, shift left=3.8, from=1-1, to=1-3]
							\arrow[iso,"\nu"'{pos=0.58,inner sep=0.75ex}, shift right=5.6, from=1-1, to=1-3]
							\arrow[bend right=41.5, null, from=1-1, to=1-4]
							\arrow[bend left=41.5, null, from=1-1, to=1-4]
							\arrow["k"', from=1-2, to=1-3]
							\arrow[bend right=39,null, from=1-2, to=1-4]
							\arrow[iso,"\alpha"'{pos=0.57,inner sep=0.7ex}, shift right=3.8, from=1-2, to=1-4]
							\arrow[iso,"\nu"{pos=0.58,inner sep=0.75ex}, shift left=5.6, from=1-2, to=1-4]
							\arrow["f"', from=1-3, to=1-4]
						\end{cd}
						is an invertible null 2-cell. By $(ii)$, since $n$ is a null morphism, there exist a null object $Z$ and an isomorphic 2-cell $\gamma$ such that 
						\begin{cd}[5][5]
							D \&\&\&\& A \\[-2ex]
							\& Z \&\& Z \& \hphantom{.}
							\arrow["n",null, from=1-1, to=1-5]
							\arrow[null,bend right=90, from=1-1, to=1-5]
							\arrow["x"', from=1-1, to=2-2]
							\arrow[equal,bend left=34, from=2-2, to=2-4]
							\arrow[null,"{\widehat{\xi^{Z}}}"'{pos=0.45,inner sep=0.4ex}, bend right=15,from=2-2, to=2-4]
							\arrow[iso,"{\xi^Z}"{pos=0.66,inner sep=0.68}, shift left=0.6, from=2-2, to=2-4]
							\arrow[iso,"{\nu}"{pos=0.585,inner sep=0.68}, shift right=4.4,from=2-2, to=2-5]
							\arrow[iso,"\gamma"{pos=0.62,inner sep=0.68ex}, shift left=6.85, start anchor={[xshift=8ex]},from=2-2, to=2-4]
							\arrow["y"', from=2-4, to=1-5]
						\end{cd}
						is an invertible null 2-cell. By the universal property of the $\N$-2-kernel, there exist $u\:Z\to K$ and an isomorphic 2-cell $\lambda\:y\aR{}k\c u$ such that
						\begin{cd}[6.3][6.3]
							{Z}\arrow[rr,shift left=2.1,phantom,"{{\xi^Z}^{-1}}"'{pos=0.42,inner sep=0.5ex,start anchor={[xshift=1ex]}}]\arrow[r,iso,shift left=1.9]\arrow[r,bend left=45,null,"{\widehat{\xi^Z}}"] \arrow[r,equal]\arrow[rrrd,""{name=C}, null, bend right=69] \arrow[rd,"{\exists u}"'{name=A}, dashed] \arrow[rrd,phantom,""{name=B}, bend left=20] \arrow[rrrd,"", null, bend left=27,start anchor={[xshift=-0.5ex,yshift=2.4ex]},end anchor={[xshift=1.5ex,yshift=0.4ex]}] \&[-1ex] {Z} \arrow[rd,"{y}"{inner sep=0.25ex}]\&[-2ex] \hphantom{.} \&[-2ex] {}\\[-4ex]
							{} \& {K} \arrow[r,"{\exists \lambda}"'{pos=0.65}, iso, shift left= 3ex, start anchor={[xshift=-8ex]}] 
							\arrow[r,"{\nu}"'{pos=0.64}, iso, shift right= 2.8ex, start anchor={[xshift=-8ex]}] 
							\arrow[r,"k"] \arrow[rr,""{name=D}, bend right=42, null] \& {A} \arrow[r,"{\nu}"'{pos=0.64}, iso, shift left= 3.3ex, start anchor={[xshift=-8ex]}] \arrow[r,"{\alpha}"'{pos=0.64}, iso, shift right= 2ex, start anchor={[xshift=-8ex]}]  \arrow[r,"f"] \& {B}
						\end{cd}
						is an invertible null 2-cell. It is then straightforward to see that we can apply the 2-dimensional universal property of $\N$-2-kernel to the isomorphic 2-cell $\Theta$ given by
						\begin{cd}[3][5]
							\&\& K \\
							D \&\hphantom{.}\&\hphantom{.}\& A \\
							\& Z \& K \& \hphantom{.}
							\arrow["k", from=1-3, to=2-4]
							\arrow["s", from=2-1, to=1-3]
							\arrow["n", bend left=5, from=2-1, to=2-4]
							\arrow[iso,"\gamma"'{pos=0.58,inner sep=0.8ex}, shift right=3.4, from=2-1, to=2-3]
							\arrow[iso,"\delta"{pos=0.58,inner sep=0.8ex}, shift left=4.6, from=2-2, to=2-4]
							\arrow["x"', from=2-1, to=3-2]
							\arrow["y"{pos=0.35,inner sep=0.21ex},  bend left=20, from=3-2, to=2-4]
							\arrow["u"', from=3-2, to=3-3]
							\arrow[iso,"\lambda"{pos=0.6,inner sep=0.9ex}, shift left=5.2, from=3-2, to=3-4,shorten <=1.3ex]
							\arrow["k"', from=3-3, to=2-4]
						\end{cd}
						We obtain that there exists a unique isomorphic 2-cell $\mu\:s\aR{}u\c x$ such that $k\ast \mu=\Theta$. So $s$ is isomorphic to a null morphism via
						\begin{cd}
							D \& Z \& Z \& K
							\arrow["x"', from=1-1, to=1-2]
							\arrow["s", bend left=35, from=1-1, to=1-4]
							\arrow[null,bend right=35, from=1-1, to=1-4]
							\arrow[iso,"\mu"{pos=0.56,inner sep=0.75ex}, shift left=5.4, from=1-1, to=1-3]
							\arrow[iso,"\nu"{pos=0.59,inner sep=0.8}, shift right=6.3, from=1-2, to=1-4]
							\arrow[equal, bend left=30,from=1-2, to=1-3]
							\arrow[null,"{\widehat{\xi^Z}}"', bend right=30, from=1-2, to=1-3]
							\arrow[iso,"{\xi^Z}"'{pos=0.73,inner sep=0.58}, from=1-2, to=1-3]
							\arrow["u"', from=1-3, to=1-4]
						\end{cd}
						It is now straightforward to conclude.
						
						We have thus proved that $(i)$ and $(ii)$ are equivalent. Dually, we obtain that also $(ii)$ and $(iii)$ are equivalent.
					\end{proof}
					
					In the last part of this section, we introduce a notion of equivalence between 2-ideals. We will show that equivalent 2-ideals have same 2-kernels and 2-cokernels. This will be an important ingredient for the proof of \thex\ref{theorchargrandistwocat}.
					
					\begin{definition}\label{defeq2ideals}
						Let $\L$ be a 2-category and let $\N=(N, \nu)$ and $\N'=(N',\nu')$ be 2-ideals of null morphisms and null 2-cells in $\L$. We say that $\N$ is \dfn{equivalent} to $\N'$ if there exist a pseudonatural equivalence
						\begin{cd}
							{\L \op \x \L}  \arrow[r,simeq, "\iota"'{inner sep=0.8ex}, xshift=-0.35ex] \arrow[r, bend left=30, "N"] \arrow[r, bend right=30, "N'"']  \& {\Cat}
						\end{cd}
						and an invertible modification 
						\begin{eqD*}
							\begin{cd}*
								{\L \op \x \L}  \arrow[r,simeq, "\iota"'{inner sep=0.8ex}, xshift=-0.35ex, shift left=2ex] \arrow[r, bend left=40, "N"{pos=0.5}] \arrow[r, "N'"'{pos=0.5, name=A}] \arrow[r, bend right=60, "{\operatorname{Hom}(-,-)}"'{pos=0.511,name=B}]   \arrow[from=A, to=B, Rightarrow, "\nu'"{pos=0.1}, shorten <= 0.2ex, shorten >=0.2ex]\& {\Cat}
							\end{cd}
							\quad 
							\aM[\Xi]{\iso}
							\quad
							\begin{cd}*
								{\L \op \x \L}   \arrow[r, bend left=40, "N"{pos=0.5}," "{name=C}]  \arrow[r, bend right=40, "{\operatorname{Hom}(-,-)}"'{pos=0.5,name=B}, " "{name=D}]   \arrow[from=C, to=D, Rightarrow, "\nu"{pos=0.5}, shorten <= 1.2ex, shorten >=1.2ex]\& {\Cat}
							\end{cd}
						\end{eqD*}
					\end{definition}
					
					We aim at proving some useful explicit characterizations of equivalence of 2-ideals.
					
					\begin{definition}\label{listofproperties}
						Let $\N=(N, \nu)$ and $\N'=(N',\nu')$ be 2-ideals in a 2-category $\L$ and consider the following conditions:
						\begin{itemize}
							\item [(A)] for every $A \anull{n} B$ in $\N$ there exist $A \anull{n'} B$ in $\N'$ and an invertible 2-cell 
							\begin{cd}[6][7.5]
								A\ar[r,null,bend right=49,"n"'{inner sep=1ex,pos=0.53}] \ar[r,null,bend left=49,"n'"{inner sep=0.8ex}]\ar[r,shift left=1ex,iso,"\Xi_{A,B,n}"'{pos=0.6, inner sep=0.8ex}] \& B; 
							\end{cd}
							\item [(B)] for every $A \anull{n'} B$ in $\N'$ there exist $A \anull{m} B$ in $\N$ and an invertible 2-cell 
							\begin{cd}[6][7.5]
								A\ar[r,null,bend right=49,"n'"'{inner sep=0.9ex,pos=0.53}] \ar[r,null,bend left=49,"m"{inner sep=0.8ex}]\ar[r,shift left=1ex,iso,"\Xi'_{A,B,n'}"'{pos=0.6, inner sep=0.8ex}] \& B. 
							\end{cd}
						\end{itemize}
						Assuming that (A) and (B) are satisfied, consider the following properties relative to the isomorphisms $\Xi$ and $\Xi'$ of properties (A) and (B):
						\begin{itemize}
							\item [(1)] for every null 2-cell $\alpha\:m\aR{}n\:A\to B$ in $\N$ the 2-cell
							\begin{cd}[6][12.5]
								A  \arrow[r,iso,shift  left=4.4ex,xshift=-0.8ex,"{\Xi_{A,B,m}}"'{pos=0.82, inner sep=-1.5ex}] \arrow[r,iso,shift right=4.2ex,xshift=-0.8ex,"{\Xi^{-1}_{A,B,n}}"'{pos=0.82, inner sep=-1.5ex}] \arrow[r,"m"{inner sep=0.4ex, pos=0.445}, ""{name=S}, bend left=22, null] \arrow[r,"m'"{inner sep=0.8ex}, bend left=90, null] 
								\arrow[r,"n"'{inner sep=0.4ex, pos=0.445}, ""'{name=T},bend right=22, null] \arrow[from=S, to=T, Rightarrow, "\alpha"{inner sep=0.9ex}, shorten <=0.8ex, shorten >=0.8ex] \arrow[r,"n'"', bend right=90, null]  \& B
							\end{cd}
							is a null 2-cell in $\N'$;
							\item [(1')] for every null 2-cell $\beta\:n'\aR{}s'\:A\to B$ in $\N'$ the 2-cell
							\begin{cd}[6][12.5]
								A  \arrow[r,iso,shift left=4.4ex,xshift=-0.8ex,"{\Xi'_{A,B,n'}}"'{pos=0.82, inner sep=-1.5ex}] \arrow[r,iso,shift right=4.2ex,xshift=-0.8ex,"{\Xi'^{-1}_{A,B,s'}}"'{pos=0.82, inner sep=-1.5ex}] \arrow[r,"n'"'{inner sep=0.2ex, pos=0.33}, ""{name=S}, bend left=22, null] \arrow[r,"m"{inner sep=0.8ex}, bend left=90, null] 
								\arrow[r,"s'"{inner sep=0.2ex, pos=0.33}, ""'{name=T},bend right=22, null] \arrow[from=S, to=T, Rightarrow, "\beta"{inner sep=0.9ex}, shorten <=0.8ex, shorten >=0.8ex] \arrow[r,"r"'{inner sep=0.5ex}, bend right=90, null]  \& B
							\end{cd}
							is a null 2-cell in $\N$;
							\item [(2)] for every null morphisms $m,n\: A \anull{}B$ in $\N$ and every null 2-cell $\beta\:m'\aR{}n'\:A\to B$ in $\N'$ (where $m'$ and $n'$ are the morphisms in $\N'$ associated to $m$ and $n$) the 2-cell
							\begin{cd}[6][12.5]
								A  \arrow[r,iso,shift left=4.4ex,xshift=-0.8ex,"{\Xi^{-1}_{A,B,m}}"'{pos=0.82, inner sep=-1.5ex}] \arrow[r,iso,shift right=4.2ex,xshift=-0.8ex,"{\Xi_{A,B,n}}"'{pos=0.82, inner sep=-1.5ex}] \arrow[r,"m'"{inner sep=0.4ex, pos=0.445}, ""{name=S}, bend left=22, null] \arrow[r,"m"{inner sep=0.8ex}, bend left=90, null] 
								\arrow[r,"n'"'{inner sep=0.4ex, pos=0.445}, ""'{name=T},bend right=22, null] \arrow[from=S, to=T, Rightarrow, "\beta"{inner sep=0.9ex}, shorten <=0.8ex, shorten >=0.8ex] \arrow[r,"n"', bend right=90, null]  \& B
							\end{cd}
							is a null 2-cell in $\N$;
							\item[(1+2)] for every 2-cell $\alpha\:m\aR{}n\:A\to B$ between null morphisms $m$ and $n$ in $\N$ the 2-cell 
							\begin{cd}[6][12.5]
								A  \arrow[r,iso,shift left=4.4ex,xshift=-0.8ex,"{\Xi_{A,B,m}}"'{pos=0.82, inner sep=-1.5ex}] \arrow[r,iso,xshift=-0.8ex,shift right=4.2ex,"{\Xi^{-1}_{A,B,n}}"'{pos=0.82, inner sep=-1.5ex}] \arrow[r,"m"{inner sep=0.4ex, pos=0.445}, ""{name=S}, bend left=22, null] \arrow[r,"m'"{inner sep=0.8ex}, bend left=90, null] 
								\arrow[r,"n"'{inner sep=0.4ex, pos=0.445}, ""'{name=T},bend right=22, null] \arrow[from=S, to=T, Rightarrow, "\alpha"{inner sep=0.9ex}, shorten <=0.8ex, shorten >=0.8ex] \arrow[r,"n'"', bend right=90, null]  \& B
							\end{cd}
							is a null 2-cell in $\N'$ if and only if $\alpha$ is a null 2-cell in $\N$;
							\item[(3)] for every pair of morphisms $(a,b)\: (A,B) \to (A',B')$, the 2-cell
							\begin{cd}[6][12.5]
								A'   \arrow[rrr,null, bend left=30] \arrow[rrr,null, bend right=70] \arrow[rrr,null,"\widehat{n}"'{inner sep=0.6ex}, bend right=32] \arrow[r,"a"'] \&[-2.5ex] A \arrow[r, iso, "\nu'"'{pos=0.66, inner sep=-1.5ex}, shift left=5.8ex] \arrow[r, iso, "\Xi^{-1}_{A',B',\widehat{n}}"'{pos=0.84, inner sep=-1.5ex}, xshift=-0.8ex,shift right=11.5ex]  \arrow[r, iso, "\nu"'{pos=0.66, inner sep=-1.5ex}, shift right=6ex] \ar[r,null,bend right=30,"n"'{inner sep=1ex,pos=0.53}] \ar[r,null,bend left=30,"n'"{inner sep=0.8ex}]\ar[r,shift left=0.11ex,iso,xshift=-0.8ex,"\Xi_{A,B,n}"'{pos=0.82, inner sep=-1.5ex}] \& B \arrow[r,"b"']\&[-2.5ex] B'; 
							\end{cd}
							is an invertible null 2-cell in $\N'$;
							\item [(4)] for every  null morphism $A \anull{n'} B$ in $\N'$ the 2-cell
							\begin{cd}[6][12.5]
								A  \arrow[r,"m",null] \arrow[r,iso,shift left=2.4ex,xshift=-1.9ex,"{{\Xi'}^{-1}_{A,B,n'}}"'{pos=0.85, inner sep=-1.5ex}] \arrow[r,iso,shift right=2.5ex,xshift=-1.9ex,"{\Xi^{-1}_{A,B,m}}"'{pos=0.82, inner sep=-1.5ex}]  
								\arrow[r,"m'"', bend right=60, null]   \arrow[r,"n'"{inner sep=0.7ex}, bend left=60, null] \& B
							\end{cd}
							is an invertible null 2-cell in $\N'$;
							\item [(4')] for every  null morphism $A \anull{m} B$ in $\N$ the 2-cell
							\begin{cd}[6][12.5]
								A  \arrow[r,"m'",null] \arrow[r,iso,shift left=2.4ex,xshift=-1.9ex,"{{\Xi}^{-1}_{A,B,m}}"'{pos=0.85, inner sep=-1.5ex}] \arrow[r,iso,shift right=2.5ex,xshift=-1.9ex,"{{\Xi'}^{-1}_{A,B,m'}}"'{pos=0.85, inner sep=-1.5ex}]  
								\arrow[r,"l"', bend right=60, null]   \arrow[r,"m"{inner sep=0.7ex}, bend left=60, null] \& B
							\end{cd}
							is an invertible null 2-cell in $\N$.
						\end{itemize}
					\end{definition}

					\begin{theorem}\label{theorchareq2ideals}
						Let $\L$ be a 2-category and let $\N=(N,\nu)$ and $\N'=(N',\nu')$ be 2-ideals of null morphisms and null 2-cells in $\L$. The following facts are equivalent:
						\begin{itemize}
							\item[(i)] $\N$ is equivalent to $\N'$;
							\item[(ii)] conditions (A) and (B) hold and the isomorphisms $\Xi$ and $\Xi'$ satisfy properties (1), (2), (3) and (4);
							\item[(iii)] conditions (A) and (B) hold and the isomorphisms $\Xi$ and $\Xi'$ satisfy properties (1+ 2), (3) and (4);
							\item[(iv)] conditions (A) and (B) hold and the isomorphisms $\Xi$ and $\Xi'$ satisfy properties (1), (1'), (3), (4) and (4').
						\end{itemize}
					\end{theorem}
					
					\begin{proof}
						(i) $\Rightarrow $ (ii).  Condition (A) holds thanks to the existence of the invertible modification $\Xi$ and the isomorphisms $\Xi'$ of condition (B) are constructed by pasting the isomorphic diagrams related by $\Xi$ (as in \defx\ref{defeq2ideals}) with the pseudoinverse of the pseudonatural equivalence $\iota$. The 2-cell of property (1) is precisely the $\iota_{A,B}(\alpha)$ and so it is a null 2-cell in $\N'$. Moreover, property (2) is precisely the fullness of the equivalence of categories $\iota_{A,B}$. The 2-cell of property (3) is then precisely given by the action of structure 2-cell  $\iota_{(a,b)}$ on the null morphism $A \anull{n} B$ in $\N$ and thus it is a null 2-cell in $\N'$. Finally, (4) is precisely given by the essential surjectivity of the equivalence of categories $\iota_{A,B}$.
						
						(ii) $\Rightarrow $ (i). (A) allows us to define a function $\iota_{A,B}: N(A,B) \to N'(A,B)$, that then extends to a functor thanks to (A) and (1). $\iota_{A,B}$  is automatically faithful, it full thanks to (2) and it is essentially surjective thanks to (B) and (4). So $\iota_{A,B}$ is an equivalence of categories for every $a,b\in \L$. Property (3) then allows us to define the structure 2-cells of $\iota$ and these are automatically natural transformations. It is then straightforward to see that $\iota$ constructed this way is automatically a pseudonatural transformation. It remains to define the invertible modification $\Xi$.  Property (A) allows us to define the invertible components $\Xi_{A,B,m}$ that form a natural transformation $\Xi_{A,B}$ by construction of the action of $\iota_{A,B}$ on morphisms. Finally, $\Xi$ satisfies the axiom of modification by construction of the structure 2-cells of $\iota$.
						
						(ii) $\Rightarrow $ (iii).  We show that properties (1) and (2) imply property (1+2). Let $\alpha\:m\aR{}n\:A\to B$ be a 2-cell between null morphisms $m$ and $n$ in $\N$. If $\alpha$ is a null 2-cell in $\N$ then clearly the 2-cell of property (1+2) is a null 2-cell in $\N'$ by property (1). Conversely, if such 2-cell is null in $\N'$ then applying property (2) to it we  precisely conclude that $\alpha$ is a null 2-cell in $\N$.
						
						(iii) $\Rightarrow $ (ii). Property (1+2) clearly implies (1). Moreover,  given null morphisms $m,n\: A \anull{}B$ in $\N$ and a null 2-cell $\beta\:m'\aR{}n'\:A\to B$ in $\N'$ (where $m'$ and $n'$ are the morphisms in $\N'$ associated to $m$ and $n$), then the 2-cell of property (2) is such that its pasting with $\Xi$ and $\Xi^{-1}$ gives the null 2-cell $\beta$ in $\N'$ and so it is a null 2-cell in $\N$ by (1+2) .
						
						(iii) $\Rightarrow $ (iv).  We show that (1+2) together with (4) implies both (1') and (4'). Let $\beta\:n'\aR{}s'\:A\to B$ be a null 2-cell in $\N'$. Thanks to (1+2), to prove that (1') holds it suffices to show that the 2-cell of (1') is a null 2-cell in $\N'$ when pasted with $\Xi_{A,B,m}$ and $\Xi^{-1}_{A,B,r}$. And this is true because the obtained 2-cell is the pasting of the null 2-cell $\beta$ in $\N'$ with the invertible null 2-cells in $\N'$ given by (4) applied to $m$ and $r$. So (1') holds. Let now $A \anull{m} B$ be a null morphism in $\N$. Thanks to (1+2), is suffices to show that the 2-cell of (4') is an invertible null 2-cell in $\N'$ once pasted with $\Xi_{A,B,m}$ and $\Xi^{-1}_{A,B,l}$. But this is guaranteed by (4) and so (4') holds.

						(iv) $\Rightarrow $ (ii). We prove that (1') and (4') imply (2).  Let $m,n\: A \anull{}B$ be null morphisms in $\N$ and let $\beta\:m'\aR{}n'\:A\to B$ be a null 2-cell in $\N'$ (where $m'$ and $n'$ are the morphisms in $\N'$ associated to $m$ and $n$). Thanks to (1') the 2-cell $\beta$ pasted with ${\Xi'}^{-1}_{A,B,m'}$ and ${\Xi'}^{-1}_{A,B,n'}$ gives a null 2-cell in $\N$. Pasting this null 2-cell in $\N$ with the invertible null 2-cells given by property (4') for $m'$ and $n'$ we obtain the 2-cell involved in property (2) and thus we conclude that it is a null 2-cell in $\N$.
					\end{proof}
					
					We now prove that equivalent 2-ideals have same 2-kernels and 2-cokernels.
					
					\begin{theorem}\label{teorequividealshavesameker}
						Let $\L$ be a 2-category and let $\N=(N,\nu)$ and $\N'=(N',\nu')$ be equivalent 2-ideals of null morphisms and null 2-cells in $\L$, via a pseudonatural equivalence $\iota$ and an invertible modification $\Xi$. Then $\N$ and $\N'$ have same 2-kernels and 2-cokernels.
						
						More precisely, if $k\:K\to A$ is the $\N$-2-kernel of a morphism $f\:A\to B$ in $\L$ with structure 2-cell
						\begin{cd}
							K\ar[rr,null,bend right=35,"n"'{inner sep=0.8ex}]\ar[rr,iso,shift right=2ex,"\alpha"'{pos=0.6}]\ar[r,"k"] \& A \ar[r,"f"] \& B
						\end{cd}
						then $k$ is also the $\N'$-2-kernel of $f$, with structure 2-cell
						\begin{cd}[6][7.5]
							K \ar[rr,null,bend right=27,"n"'{inner sep=0.2ex, pos=0.42}] \ar[rr,null,bend right=85,"n'"'{inner sep=0.8ex}]\ar[rr,iso,shift right=2ex,"\alpha"'{pos=0.6, inner sep=-0.5ex}] \ar[rr,iso,shift right=5.6ex,xshift=-0.1ex,"\Xi^{-1}_{K,B,n}"'{pos=0.7, inner sep=-0.8ex}]\ar[r,"k"] \& A \ar[r,"f"] \& B
						\end{cd}
					\end{theorem}
					\begin{proof}
						Let $l\:L\to A$ be a morphism such that $f\c l$ is isomorphic to a null morphism $s'$ in $\N'$, via an isomorphic 2-cell $\beta$. Pasting with $\Xi'_{L,B,s'}$, we obtain that $f\c l$ is also isomorphic to a null morphism $r$ in $\N$. Then, by the universal property of the $\N$-2-kernel, there exist a morphism $L \aar{u} K$ and an isomorphic 2-cell $\gamma\: l \aR{} k \c u$ such that the pasting
						\begin{cd}[6.3][6.3]
							{L} \ar[rrrd,null,bend left=65,"r",start anchor = north east, end anchor = north]\arrow[rrrd,""{name=C}, null, bend right=69] \arrow[rd,"{\exists u}"'{name=A}, dashed] \arrow[rrd,"l"{name=B}, bend left=20] \arrow[rrrd,"{s'}"'{pos=0.57,inner sep=0.17ex}, null, bend left] \&[-2ex] {} \&[-2ex] {} \&[-2ex] {}\\[-4ex]
							{} \& {K} \ar[rr,iso,shift left=12.4,"{\Xi'_{L,B,s'}}"{pos=0.68,inner sep=0.75ex}] \arrow[r,"{\exists \gamma}"'{pos=0.65}, iso, shift left= 3ex, start anchor={[xshift=-8ex]}] 
							\arrow[r,"{\nu}"'{pos=0.65}, iso, shift right= 2.8ex, start anchor={[xshift=-8ex]}] 
							\arrow[r,"k"] \arrow[rr,""{name=D}, bend right=42, null] \& {A} \arrow[r,"{\beta}"'{pos=0.65}, iso, shift left= 3.3ex, start anchor={[xshift=-8ex]}] \arrow[r,"{\alpha}"'{pos=0.65}, iso, shift right= 2ex, start anchor={[xshift=-8ex]}]  \arrow[r,"f"] \& {B}
						\end{cd}
						is an invertible null 2-cell in $\N$. Since $\N$ and $\N'$ are equivalent 2-ideals, by \thex\ref{theorchareq2ideals} properties (1), (2), (3) and (4) of \def\ref{listofproperties} all hold. Applying property (1) to the invertible null 2-cell in $\N$ above and pasting with the invertible null 2-cells of properties (4) and (3), we obtain that
						\begin{cd}[7][10]
							{L}  \arrow[rrrd,""{name=C}, null, bend right=90, ""]  \arrow[rd,"{u}"'{name=A}] \arrow[rrd,"z"{name=B}, bend left=20] \arrow[rrrd,"s'", null, bend left] \&[-2ex] {} \&[-2ex] {} \&[-2ex] {}\\[-4ex]
							{} \arrow[r,iso, shift right=3ex, xshift=2ex,"\nu'"'{pos=0.6}] \& {K} \arrow[r,"{ \gamma}"'{pos=0.6}, iso, shift left= 3ex, start anchor={[xshift=-10ex]}] 
							\arrow[r,"k"] \arrow[rr,""{name=D}, bend right=30, null] \arrow[rr,""{name=D}, bend right=66, null]\& {A} \arrow[r,"{\beta}"'{pos=0.6}, iso, shift left= 3.3ex, start anchor={[xshift=-10ex]}] \arrow[r,"{\alpha}"'{pos=0.6, inner sep=-0.8ex}, iso, shift right= 1.8ex, start anchor={[xshift=-10ex]}] \arrow[r,xshift=-0.9ex,"{\Xi^{-1}_{K,B,n}}"'{pos=0.72, inner sep=-3.8ex}, iso, shift right= 6ex, start anchor={[xshift=-10ex]}] 
							\arrow[r,"f"] \& {B}
						\end{cd}
						is an invertible null 2-cell in $\N'$. So that the 1-dimensional universal property of $k$ being the $N'$-2-kernel of $f$ holds.
						
						Consider now morphisms $u,v\:L\to K$ and a 2-cell $\lambda\:k\c u\aR{}k\c v$ such that $f\ast \lambda$, pasted with $\alpha,\alpha^{-1},\Xi^{-1}_{K,B,n},\Xi_{K,B,n}$ and copies of $\nu'$, gives a null 2-cell in $\N'$. Applying properties (3) and (2) of \defx\ref{listofproperties}, it is straightforward to show that $f\ast \lambda$ also gives, when pasted with $\alpha,\alpha^{-1}$ and copies of $\nu$, a null 2-cell in $\N$. Then, by the 2-dimensional universal property of the $\N$-2-kernel, there exists a unique 2-cell $\mu\:u\aR{}v$ such that $k\ast \mu=\lambda$. So that the 2-dimensional universal property of the $\N'$-2-kernel holds.
					\end{proof}

					\section{Grandis exact 2-categories}
					
					It is shown in \cite{JanWei16} that Grandis exact categories can be equivalently characterized as categories equipped with a proper factorization system such that the opfibration of subobjects relative to the factorization system is isomorphic to the fibration of relative quotients. In this section, we develop a 2-dimensional analogue of this fibrational approach to exactness. As an outcome, we propose a 2-dimensional notion of Grandis exact category. In particular we obtain a 2-dimensional notion of Puppe exact category, in the pointed case, which we will compare in the next section with existing similar notions developed by Dupont and Nakaoka.
					
					\begin{remark}
						We will need the notion of factorization system $(\E,\M)$ on a 2-category, studied by Dupont and Vitale in \cite{DupVit03}. We will also use their notion of $(1,1)$-proper factorization system on a 2-category. This means that every morphism in $\E$ is cofaithful and every morphism in $\M$ is faithful.
					\end{remark}
					
					We will then need the following notion of 2-dimensional fibration, which is a weaker version of the bicategorical fibrations studied by Bakovi\'{c} in \cite{Bak22}. More precisely, we will only ask weak 2-fibrations to be locally an isofibration rather than a fibration.
					
					\begin{definition}
						Let $P\:\mathfrak{K}\to {\L}$ be a 2-functor. We call $P$ a \dfn{weak 2-fibration} if it satisfies the following conditions:
						\begin{itemize}
							\item[-] for every $E\in \mathfrak{K}$ and every morphism $f\:L\to P(E)$ in $\L$, there exists a 2-cartesian lifting of $f$ to $E$, in the bicategorical sense described by Bakovi\'{c} in \cite[\defx{4.6}]{Bak22};
							\item[-] $P$ is locally an isofibration, i.e.\ for every $E,E'\in \mathfrak{K}$ the induced functor $\HomC{\mathfrak{K}}{E}{E'}\to \HomC{\L}{P(E)}{P(E')}$ between hom-categories is an isofibration.
						\end{itemize}
						
						Dually, we can define \dfn{weak 2-opfibrations}.
					\end{definition}
					
					We generalize to dimension 2 the important fact that subobjects relative to a factorization system form an opfibration and quotients relative to a factorization system form a fibration. For this, we will need the notion of pseudo arrow category of a 2-category $\L$, whose objects are arrows in $\L$, morphisms are squares in $\L$ filled with an isomorphism and 2-cells are pairs of two 2-cells in $\L$ coherent with the isomorphisms filling the squares.
					
					\begin{theorem}\label{teorfibofsubobj}
						Let $(\E,\M)$ be a factorization system on a 2-category $\L$. Consider $\Esq$ and $\Msq$ the full sub-2-categories of the pseudo arrow category of $\L$ on morphisms in $\E$ and $\M$ respectively. Then the domain 2-functor $\operatorname{dom}\:\Esq\to \L$ is a weak 2-fibration, called \dfn{the weak 2-fibration of 2-quotients} relative to $(\E,\M)$.
						
						Dually, the codomain 2-functor $\operatorname{cod}\:\Msq\to \L$ is a weak 2-opfibration, called \dfn{the weak 2-opfibration of 2-subobjects} relative to $(\E,\M)$.
					\end{theorem}
					\begin{proof}
						We prove that $\operatorname{cod}\:\Msq\to \L$ is a weak-2-opfibration. Then $\operatorname{dom}\:\Esq\to \L$ will be a weak 2-fibration, dually. Of course $\operatorname{cod}$ is a 2-functor, being a restriction of the codomain 2-functor from the pseudo arrow category. Consider $m\:A\tom B$ in $\M$ and a morphism $f\:B\to C$ in $\L$, we construct the 2-cocartesian lifting of $f$ to $m$ using the factorization of $f\c m$:
						\begin{cd}
							A \ar[d,"m"',rightarrowtail,""{name=A}] \ar[r,"{\ell^{f\c m}}",twoheadrightarrow] \& Q\ar[d,"{r^{f\c m}}",""'{name=B},rightarrowtail] \\
							B \ar[r,"f"']\& C
							\ar[from=A,to=B,iso,"{\theta_{f\c m}^{-1}}"{inner sep=1ex},shift right=0.8ex]
						\end{cd}
						We show that such morphism in $\Msq$ is indeed 2-cocartesian. So consider $(z,h,\phi)\:m\to n$ in $\Msq$ with $n\:X\tom Y$, a morphism $g\:C\to Y$ in $\L$ and an isomorphic 2-cell $\xi\:h\to g\c f$ in $\L$. Then the diagonal lifting property of the factorization system $(\E,\M)$ induces a fill-in for the square on the left below:
						\[\begin{tikzcd}
							A && X \\
							& B \\
							Q & C & Y
							\arrow["z", from=1-1, to=1-3]
							\arrow["m", tail, from=1-1, to=2-2]
							\arrow[""{name=0, anchor=center, inner sep=0}, "{\ell^{f\c m}}"', two heads, from=1-1, to=3-1]
							\arrow[""{name=1, anchor=center, inner sep=0}, "n", tail, from=1-3, to=3-3]
							\arrow[""{name=2, anchor=center, inner sep=0}, "f"', from=2-2, to=3-2]
							\arrow[""{name=3, anchor=center, inner sep=0}, "h", from=2-2, to=3-3]
							\arrow["{r^{f\c m}}"', tail, from=3-1, to=3-2]
							\arrow["g"', from=3-2, to=3-3]
							\arrow["{\theta_{f\c m}}"'{inner sep=1ex}, shift right=4, shorten <=5pt, iso, from=0, to=2-2]
							\arrow["\phi"{inner sep=1ex,pos=0.4}, shift left=3, shorten >=5pt, iso, from=2-2, to=1,start anchor={[xshift=-1.5ex]}]
							\arrow["\xi"'{inner sep=1ex}, shift right, shorten <=3pt, shorten >=3pt, iso, from=2, to=3]
						\end{tikzcd} \quad \h[2] = \quad
						\begin{tikzcd}
							A && X \\
							\\
							Q & C & Y
							\arrow["z", from=1-1, to=1-3]
							\arrow[""{name=0, anchor=center, inner sep=0}, "{\ell^{f\c m}}"', two heads, from=1-1, to=3-1]
							\arrow[""{name=1, anchor=center, inner sep=0}, "n", tail, from=1-3, to=3-3]
							\arrow[""{name=2, anchor=center, inner sep=0}, "d"', from=3-1, to=1-3]
							\arrow["{r^{f\c m}}"', tail, from=3-1, to=3-2]
							\arrow["g"', from=3-2, to=3-3]
							\arrow["\alpha"{inner sep=1ex}, shift left=3, iso, from=0, to=2]
							\arrow["\beta"'{inner sep=1ex}, shift right=3, iso, from=2, to=1]
						\end{tikzcd}
						\]
						Whence $(\alpha,\xi)$ is an isomorphic 2-cell in $\Msq$ between $\beta\c \theta_{f\c m}^{-1}$ and $\phi$.
						
						Consider now $(z,h,\phi),(z',h',\phi')\:m\to n$ in $\Msq$, $g,g'$ and $\xi,\xi'\:h\to g\c f$. And take a 2-cell $(\sigma,\tau)$ from $\phi$ to $\phi'$ in $\Msq$ and a 2-cell $\rho\:g\aR{}g'$ such that 
						\begin{eqD*}
							\begin{cd}*
								B \ar[rr,iso,shift left=1.85ex,"\xi"{inner sep=1ex},start anchor={[xshift=-2.3ex]}]\ar[r,"f"']\ar[rr,bend left=40,"h"]\& C \ar[r,bend left=27,"g",""{name=A}]\ar[r,bend right=27,"g'"',""'{name=B}] \& Y
								\ar[from=A,to=B,"\rho",Rightarrow,shorten <=0.5ex,shorten >=0.5ex]
							\end{cd}\quad \h[2] = \quad
							\begin{cd}*[4.5]
								B \ar[rr,iso,shift right=3.5ex,"\xi'"'{inner sep=1ex,pos=0.5}] \ar[rd,"f"']\ar[rr,bend left=18,"h",""'{name=A}]\ar[rr,bend right=18,"h'"'{pos=0.35,inner sep=0.1ex},""{name=B}]\&\& Y \\
								\& C \ar[ru,"g'"']
								\ar[from=A,to=B,"\tau",Rightarrow,shorten <=0ex,shorten >=0ex]
							\end{cd}
						\end{eqD*}
						Consider then also $d,\alpha,\beta$ associated to $\phi,g,\xi$ as above, and $d',\alpha',\beta'$ analogously associated to $\phi',g',\xi'$. We need to show that there exists a unique 2-cell $(\lambda,\rho)$ in $\Msq$ from $(d,g,\beta)$ to $(d',g',\beta')$ such that 
						\begin{eqD*}
							\begin{cd}*
								A \ar[rr,iso,shift left=1.85ex,"\alpha"{inner sep=1ex},start anchor={[xshift=-2.3ex]}]\ar[r,"{\ell^{f\c m}}"']\ar[rr,bend left=40,"z"]\& Q \ar[r,bend left=27,"d",""{name=A}]\ar[r,bend right=27,"d'"',""'{name=B}] \& X
								\ar[from=A,to=B,"\lambda",Rightarrow,shorten <=0.5ex,shorten >=0.5ex]
							\end{cd}\quad \h[2] = \quad
							\begin{cd}*[4.5]
								A \ar[rr,iso,shift right=3.5ex,"\alpha'"'{inner sep=1ex,pos=0.5}] \ar[rd,"{\ell^{f\c m}}"'{inner sep=0.1ex}]\ar[rr,bend left=18,"z",""'{name=A}]\ar[rr,bend right=18,"z'"'{pos=0.35,inner sep=0.1ex},""{name=B}]\&\& X \\
								\& Q \ar[ru,"d'"']
								\ar[from=A,to=B,"\sigma",Rightarrow,shorten <=0ex,shorten >=0ex]
							\end{cd}
						\end{eqD*}
						Notice that $(\sigma,\rho\ast r^{f\c m})$ is a 2-cell from the square formed by $\alpha$ and $\beta$ to the square formed by $\alpha'$ and $\beta'$. By the 2-dimensional diagonal lifting property, there exists a unique 2-cell $\lambda\:d\aR{}d'$ that works. We conclude that $(\ell^{f\c m},f,\theta_{f\c m}^{-1})$ is a 2-cartesian lifting of $f$ to $m$.
						
						It remains to prove that $\operatorname{cod}\:\Msq\to \L$ is locally an isofibration. So consider a morphism
						\begin{cd}
							A \ar[d,"m"',rightarrowtail,""{name=A}] \ar[r,"{a}"] \& A'\ar[d,"{m'}",""'{name=B},rightarrowtail] \\
							B \ar[r,"b"']\& B'
							\ar[from=A,to=B,iso,"\phi"{inner sep=1ex},shift right=0.8ex]
						\end{cd}
						and an isomorphic 2-cell $\beta\:b\aR{}b'\:B\to B'$. It is straightforward to show that we can lift $\beta$ to an isomorphic 2-cell in $\Msq$ with domain $(a,b,\phi)$.
					\end{proof}
					
					We are now ready to prove the main theorem of this section, which provides a fibrational approach to define \predfn{Grandis exact 2-categories}.
					
					\begin{theorem}\label{theorchargrandistwocat}
						Let $\L$ be a 2-category. The following conditions are equivalent:
						\begin{enumT}
							\item There exists a $(1,1)$-proper factorization system $(\E,\M)$ on the 2-category $\L$ and a biequivalence over $\L$
							\begin{cd}
								\Esq \ar[rd,"\operatorname{dom}"'{inner sep=0.2ex}]\ar[rr,bend left=16,"K"]\ar[rr,simeq]\&\& \Msq \ar[ld,"\operatorname{cod}"{inner sep=0.2ex}]\ar[ll,bend left=16,"C"]\\
								\& \L
							\end{cd}
							between the weak 2-fibration of 2-quotients relative to $(\E,\M)$ and the weak 2-opfibration of 2-subobjects relative to $(\E,\M)$, with $K,C$ normal pseudofunctors.
							\item $\L$ has a 2-ideal $\N$ of null morphisms and null 2-cells such that
							\begin{itemize}
								\item[-] $\L$ has all $\N$-2-kernels and $\N$-2-cokernels;
								\item[-] $\N$ is a closed 2-ideal;
								\item[-] every $\N$-2-kernel $m$ is an $\N$-2-kernel of its $\N$-2-cokernel $c_m$, with structure isomorphic 2-cell between $c_m\c m$ and a null morphism given up to an invertible null 2-cell by the $\N$-2-cokernel $c_m$ of $m$; and dually for $\N$-2-cokernels;
								\item[-] every morphism $f$ in $\L$ factorizes up to isomorphism as an $\N$-2-cokernel followed by an $\N$-2-kernel
								\begin{cd}[0]
									A \ar[rr,iso,shift right=0.75 ex]\ar[rr,bend left=10,"f"]\ar[rd,twoheadrightarrow,"2-\operatorname{coker}"'{pos=0.7}]\&\& B \\
									\& Q \ar[ru,rightarrowtail,"2-\operatorname{ker}"'{pos=0.3}]
								\end{cd}
							\end{itemize}
						\end{enumT}
					\end{theorem}
					\begin{proof}
						We prove $(i)\aR{}(ii)$. We define a 2-ideal $\N$ of null morphisms and null 2-cells in $\L$ as follows. Null morphisms are precisely those morphisms $n\:A\to B$ that factor through $K(\id{B})$, which we denote as $K(\id{B})\aar{k_{\id{B}}}B$:\v[-2]
						\begin{cd}[4]
							\& K(\id{B})\ar[d,rightarrowtail,"{k_{\id{B}}}"] \\
							A \ar[ru,dashed,"{\exists \overline{n}}"] \ar[r,"n"']\& B
						\end{cd}
						Null 2-cells between null morphisms $n=k_{\id{B}}\c \overline{n}\:A\anull{} B$ and $n'=k_{\id{B}}\c \overline{n'}\:A\anull{} B$ are precisely those 2-cells $\mu$ that factor through $K(\id{B})$:
						\begin{cd}[6][8]
							\& K(\id{B})\ar[d,rightarrowtail,"{k_{\id{B}}}"] \ar[d,dashed, Rightarrow,shift right=8ex,shorten <=1.5ex,shorten >=2ex,"\exists \overline{\mu}"{pos=0.4}]\ar[d, Rightarrow,shift right=7ex,shorten <=6.65ex,shorten >=-3.2ex,"{\mu}"{pos=1.31}]\\
							A \ar[ru,bend left=23,"{ \overline{n}}"{pos=0.7}] \ar[ru,bend right=15,"{ \overline{n'}}"'{inner sep=0.25ex, pos=0.8}]\ar[r,bend left=20,"n"{pos=0.67}]\ar[r,bend right=20,"n'"'{pos=0.67}]\& B
						\end{cd}
						Notice that $\overline{\mu}$ is unique because $K(\id{B})\in \M$ is faithful, as $(\E,\M)$ is $(1,1)$-proper. The idea behind this construction is that we call null those morphisms that make the largest subobject of $A$, which is $\id{A}$, factor through the ``smallest subobject" of $B$, which is given by the image through $K$ of the smallest quotient of $B$ which is $\id{B}$. We will show below that the 2-ideal constructed dually using $C(\id{A})$ is equivalent to the one above, so that 2-kernels and 2-cokernels are the same for the two 2-ideals.
						
						We show that $\N$ is a 2-ideal of null morphisms and null 2-cells in $\L$. Given $A'\aar{a} A$, $B\aar{b}B'$ and a null morphism $A\anull{n} B$, we construct the null replacement of $b\c n\c a$ with the following pasting:
						\begin{cd}[4]
							\& \& K(\id{B})\ar[r,"{\widehat{K(\id{}^b)}}"]\ar[d,rightarrowtail,"{k_{\id{B}}}",""{name=A}] \& K(\id{B'})\ar[d,rightarrowtail,"{k_{\id{B'}}}",""'{name=B}]\\
							A'\ar[r,"a"'] \& A \ar[ru,dashed,"{\exists \overline{n}}"] \ar[r,"n"']\& B \ar[r,"b"'] \& B'
							\ar[from=A,to=B,iso,"{K(\id{}^b)}"{inner sep=1ex},shift right=1ex]
						\end{cd}
						where $(\widehat{K(\id{}^b)},b,{K(\id{}^b)})$ is given by applying $K$ to the morphism $\id{}^b:=(b,b,\id{})$ in $\Esq$. It is straightforward to prove that the axioms of 2-ideal listed in \thex\ref{theorchar2ideal} are satisfied. We thus have a 2-ideal $\N=(N,\nu)$ in $\L$.
						
						We show that the 2-ideal $\N'$ constructed dually to $\N$ using $C(\id{A})$ rather than $K(\id{B})$ is equivalent to the one above (in the sense of \defx\ref{defeq2ideals}). Let $s'\:A\to B$ be a null morphism in $\N'$, i.e.\ a morphism $s'$ such that
						\begin{cd}[5][5]
							A \ar[r,"{s'}"] \ar[d,two heads,"{c_{\id{A}}}"'] \& B \ar[d,equal] \\
							C(\id{A}) \ar[r,"{\exists \widetilde{s'}}"'] \& B.
						\end{cd}
						We can apply $K$ to the identity 2-cell $\id{}^{s'}$ above, seen as a morphism in $\Esq$, and paste with the unit $\eta$ of the biadjunction to obtain the isomorphic 2-cell
						\begin{cd}[1.5][5]
							A \ar[dd,equal]\ar[r,"{\widehat{\eta_{\id{A}}}}"] \& K(c_{\id{A}}) \ar[dd,"{k_{c_{\id{A}}}}"] \ar[r,"{\widehat{K(\id{}^{s'})}}"] \& K(\id{B}) \ar[dd,"{k_{\id{B}}}"] \\
							\hphantom{.}\ar[r,iso,"{\eta_{\id{A}}}"'{inner sep=1ex}] \& \hphantom{.} \ar[r,iso,"{K(\id{}^{s'})}"'{inner sep=1ex}]\& \hphantom{.} \\
							A \ar[r,equal] \& A \ar[r,"{s'}"']\& B.
						\end{cd}
						Such 2-cell is an isomorphic 2-cell $\Xi_{A,B,s'}$ from a null morphism in $\N$ to $s'$, so that condition (A) of \defx\ref{listofproperties} is satisfied. Dually, given a null morphism $n$ in $\N$, we can use $C$ and the counit $\epsilon$ of the biadjunction to construct an isomorphic 2-cell $\Xi'_{A,B,n}$ from a null morphism in $\N'$ to $n$, so that also condition (B) is satisfied. In order to show that property (1) of \defx\ref{listofproperties} holds, we can view a null 2-cell $\tau\:l'\aR{}s'$ in $\N'$, i.e.\ a 2-cell that factors through $c_{\id{A}}$, as a 2-cell in $\Esq$ from $\id{}^{l'}$ to $\id{}^{s'}$ (where $\id{}^{l'}$ is analogous to $\id{s'}$, drawn above, for $l'$ in place of $s'$). Applying $K$ to such 2-cell in $\Esq$, we get a 2-cell in $\Msq$ with second component $\tau$, which shows that property (1) holds. Dually, we get that also (1') holds. It is then straightforward to prove properties (3) and (4), whence also (4') will hold as it is dual to (4), using similar arguments to the one above. The strategy is always to find a good 2-cell in $\Esq$ and apply $K$ to it. Thanks to \thex\ref{theorchareq2ideals}, we conclude that the 2-ideals $\N$ and $\N'$ are equivalent. Whence 2-kernels and 2-cokernels are the same for the two 2-ideals, by \thex\ref{teorequividealshavesameker}. This fact will be useful later in this proof.
						
						Consider now $f\:A\to B$ in $\L$. Since $(\E,\M)$ is a factorization system on the 2-category $\L$, $f$ factorizes up to an isomorphic 2-cell $\theta_f$ as $m\c e$ with $e\in \E$ and $m\in \M$. We prove that the morphism $K(e)\in \M$, denoted as $K(e)\aar{k_e} A$, is an $\N$-2-kernel of $f$. Notice that $\theta_f$ gives the following morphism in $\E$:
						\begin{cd}[5.5][5.5]
							A\ar[r,"f"]\ar[d,"e"'{inner sep=0.25ex},twoheadrightarrow,""{name=A}]\& B\ar[d,equal,""'{name=B}]\\
							Q\ar[r,"m"'{inner sep=0.25ex}] \& B
							\ar[from=A,to=B,iso,shift right=0.8ex,"{\theta_f}"{inner sep=0.75ex}]
						\end{cd}
						Then applying $K$ to it we obtain the structure isomorphic 2-cell for the would be $\N$-2-kernel $k_e$:
						\begin{cd}[0][5]
							\& A \ar[rd,"f"]  \\
							K(e) \ar[rr,shift right=1.4ex,iso,"{K(\theta_f)^{-1}}"{inner sep=1ex,pos=0.6}] \ar[rd,"{\widehat{K(\theta_f)}}"'{inner sep=0.1ex,pos=0.6}]\ar[ru,"{k_e}"] \& \& B \\
							\& K(\id{B}) \ar[ru,"{k_{\id{B}}}"'{pos=0.4}] 
						\end{cd}
						We first show the 1-dimensional part of the universal property of the $\N$-2-kernel. So consider $z\:Z\to A$ and an isomorphic 2-cell
						\begin{cd}[5][5]
							Z \ar[r,"h"] \ar[d,"z"',""{name=A}]\& K(\id{B}) \ar[d,"{k_{\id{B}}}",""'{name=B}] \\
							A \ar[r,"f"']\& B
							\ar[iso,from=A,to=B,shift right=0.8ex,"\beta"{inner sep=1ex}]
						\end{cd}
						We need to construct a morphism $u\:Z\to K(e)$ and an isomorphic 2-cell $\gamma\:z\to k_e\c u$, which we can view as a morphism
						\begin{cd}[5][5]
							Z \ar[r,"u"] \ar[d,equal,""{name=A}]\& K(e) \ar[d,"{k_{e}}",rightarrowtail,""'{name=B}] \\
							Z \ar[r,"z"']\& A
							\ar[iso,from=A,to=B,shift right=0.8ex,"{\gamma^{-1}}"{inner sep=1ex}]
						\end{cd}
						in $\Msq$. The biequivalence over $\L$ gives an equivalence of categories
						$$\HomC{\Msq}{\id{z}}{k_e}\simeq \HomC{\Esq}{c_{\id{z}}}{e}$$
						over $\HomC{\L}{Z}{A}$. So it suffices to construct a morphism $c_{\id{z}}\to e$ in $\Esq$. By the diagonal lifting property of the factorization system $(\E,\M)$ on $\L$, we obtain a fill-in for the isomorphic 2-cell on the left below:
						\[\begin{tikzcd}
							Z & A & Q \\[-2ex]
							& B & {} \\
							{C(\id{z})} & {C(k_{\id{B}})} & B
							\arrow["z", from=1-1, to=1-2]
							\arrow[""{name=0, anchor=center, inner sep=0}, "{c_{\id{z}}}"', two heads, from=1-1, to=3-1]
							\arrow["e", two heads, from=1-2, to=1-3]
							\arrow["f", from=1-2, to=2-2]
							\arrow["m", tail, from=1-3, to=3-3]
							\arrow["{\theta_f^{-1}}"{inner sep=1ex},iso, from=2-2, to=2-3]
							\arrow[""{name=1, anchor=center, inner sep=0}, "{c_{k_{\id{B}}}}"', from=2-2, to=3-2]
							\arrow[""{name=2, anchor=center, inner sep=0}, from=2-2, to=3-3,equal]
							\arrow["{\widehat{C(\beta)}}"', from=3-1, to=3-2]
							\arrow["{\widehat{\epsilon_{\id{B}}}}"', from=3-2, to=3-3]
							\arrow["{C(\beta)}"'{inner sep=1ex}, shorten <=7pt, iso, from=0, to=2-2]
							\arrow["{\epsilon_{\id{B}}}"'{inner sep=1ex}, shorten <=4pt, shorten >=4pt, iso, from=1, to=2]
						\end{tikzcd}\quad \h[2] = \quad
						\begin{tikzcd}
							Z & A & Q \\
							&& {} \\
							{C(\id{z})} & {C(k_{\id{B}})} & B
							\arrow["z", from=1-1, to=1-2]
							\arrow["{\exists \delta}"{inner sep=1ex}, iso, shift right=5ex, from=1-1, to=1-2]
							\arrow["{c_{\id{z}}}"', two heads, from=1-1, to=3-1]
							\arrow["e", two heads, from=1-2, to=1-3]
							\arrow["m", tail, from=1-3, to=3-3]
							\arrow["{\exists v}"{description}, from=3-1, to=1-3]
							\arrow["{\widehat{C(\beta)}}"', from=3-1, to=3-2]
							\arrow["{\widehat{\epsilon_{\id{B}}}}"', from=3-2, to=3-3]
							\arrow["{\exists\delta'}"'{inner sep=1ex},iso, shift left=5ex, from=3-2, to=3-3]
						\end{tikzcd}
						\]
						$\delta$ is a morphism $c_{\id{Z}}\to e$ in $\Esq$, which then corresponds to the desired morphism\linebreak $(u,z,\gamma^{-1})\:\id{z}\to k_e$ in $\Msq$. Precisely, $(u,z,\gamma^{-1})$ is given by $K(\delta)\c \eta_{\id{Z}}$. It is then straightforward to prove that this works, making the associated pasting into an invertible null 2-cell.
						
						For the 2-dimensional universal property of the $\N$-2-kernel, consider $u,v\:Z\to K(e)$ and $\lambda\:k_e\c u\aR{}k_e\c v$ such that the associated 2-cell $\Xi^{\lambda}$ is a null 2-cell. We need to find a 2-cell $\mu\:u\aR{} v$ such that $k_e\ast \mu=\lambda$. Equivalently, we need $\mu$ such that $(\mu,\lambda)$ is a 2-cell in $\Msq$
						$$\text{from }\quad \begin{cd}*[5][5]
							Z \ar[r,"u"]\ar[d,equal,""{name=A}] \& K(e) \ar[d,"k_e",rightarrowtail,""'{name=B}]\\
							Z \ar[r,"{k_e\c u}"'] \& A
							\ar[from=A,to=B,draw=none,"{\id{}^u}"{description}]
						\end{cd} \quad \text{ to } \quad \begin{cd}*[5][5]
							Z \ar[r,"v"]\ar[d,equal,""{name=A}] \& K(e) \ar[d,"k_e",rightarrowtail,""'{name=B}]\\
							Z \ar[r,"{k_e\c v}"'] \& A
							\ar[from=A,to=B,draw=none,"{\id{}^v}"{description}]
						\end{cd}$$
						The equivalence of categories
						$$\HomC{\Msq}{\id{z}}{k_e}\simeq \HomC{\Esq}{c_{\id{z}}}{e}$$
						over $\HomC{\L}{Z}{A}$ gives a bijection on homsets
						$$\HomC{\HomC{\Msq}{\id{z}}{k_e}}{\id{}^u}{\id{}^v}\aiso{} \HomC{\HomC{\Esq}{c_{\id{z}}}{e}}{\eps_e\c C(\id{}^u)}{\eps_e\c C(\id{}^v)}$$
						over $\HomC{\HomC{\L}{Z}{A}}{k_e\c u}{k_e\c v}$. So it suffices to construct a 2-cell $\sigma\:\widehat{\epsilon_e}\c \widehat{C(\id{}^u)}\aR{}\widehat{\epsilon_e}\c \widehat{C(\id{}^v)}$ such that $(\lambda,\sigma)$ is a 2-cell in $\Msq$ from $\eps_e\c C(\id{}^u)$ to $\eps_e\c C(\id{}^v)$. Since $\Xi^{\lambda}$ is a null 2-cell by assumption, there exists $\overline{\lambda}\:\widehat{K(\theta_f)}\c u \aR{} \widehat{K(\theta_f)}\c v$ such that $\Xi^{\lambda}=k_{\id{B}}\ast \overline{\lambda}$. Thus $(\overline{\lambda},f\ast \lambda)$ is a 2-cell in $\Msq$ from $K(\theta_f)\c \id{}^u$ to $K(\theta_f)\c \id{}^v$. Considering $C(\overline{\lambda},f\ast \lambda)$, one can obtain a 2-cell $(e\ast \lambda,\tau)$ from the filled-in square
						\[\begin{tikzcd}
							Z & A & Q \\
							& {C(k_e)} & {} \\
							{C(\id{z})} && B
							\arrow["{k_e\c u}", from=1-1, to=1-2]
							\arrow[""{name=0, anchor=center, inner sep=0}, "{c_{\id{z}}}"', two heads, from=1-1, to=3-1]
							\arrow["e", two heads, from=1-2, to=1-3]
							\arrow[""{name=1, anchor=center, inner sep=0}, "{c_{k_e}}"', from=1-2, to=2-2]
							\arrow["m", tail, from=1-3, to=3-3]
							\arrow[""{name=2, anchor=center, inner sep=0}, "{\widehat{\epsilon_e}}"', from=2-2, to=1-3]
							\arrow[""{name=3, anchor=center, inner sep=0}, "{\widehat{C(\id{}^u)}}"', from=3-1, to=2-2]
							\arrow["{m\c \widehat{\epsilon_e}\c \widehat{C(\id{}^u)}}"'{inner sep=0.1ex}, from=3-1, to=3-3]
							\arrow["{C(\id{}^u)}"{inner sep=1ex}, shorten <=5pt, iso, from=0, to=2-2]
							\arrow["{\epsilon_e}"{inner sep=1ex,pos=0.55}, shorten <=3pt, shorten >=3pt, iso, from=1, to=2]
							\arrow["{\id{}}"{description, pos=0.6}, shift right=4, draw=none, from=3, to=2-3]
						\end{tikzcd}\]
						to the analogous one with $v$ in place of $u$. By the 2-dimensional part of the diagonal lifting property, we obtain the desired 2-cell $\sigma$ that works. So we obtain $\mu$ such that $k_e\ast \mu=\lambda$. And such $\mu$ is unique because $k_e$ is faithful, as $(\E,\M)$ is $(1,1)$-proper. So $K(e)=(K(e)\aar{k_e} A)$ is an $\N$-2-kernel of $f$.
						
						Dually, $C(m)=(B\aar{c_m} C(m))$ is an $\N$-2-cokernel of $f$. Indeed, it is surely a 2-cokernel of $f$ with respect to the 2-ideal $\N'$ constructed using $C(\id{A})$, and we observed above that the 2-ideal $\N'$ is equivalent to $\N$. So that $C(m)$ is also an $\N$-2-cokernel of $f$. We have thus proved that $\L$ has all $\N$-2-kernels and $\N$-2-cokernels. Notice then that, since $e$ has trivial $(\E,\M)$-factorization $e=\id{}\c e$, we have that $k_e$ is also an $\N$-2-kernel of $e$. And dually $c_m$ is also an $\N$-2-cokernel of $m$. In particular we also have that $c_{k_{e}}$ is an $\N$-2-cokernel of the $\N$-2-kernel of $e$.
						
						Since $K$ and $C$ give a biequivalence over $\L$, there exists an equivalence $\epsilon_e$ in $\Esq$ between $C(K(e))$ and $e$. This gives in particular an equivalence $\widehat{\epsilon}_e\:C(k_e)\simeq Q$ in $\L$ and an isomorphic 2-cell
						\begin{cd}[4.5][4.5]
							A \ar[r,iso,shift right=5ex,"{\epsilon_e}"{inner sep =1ex}]\ar[d,"c_{k_e}"']\ar[r,equal] \& A \ar[d,"{e}"] \\
							C(k_e) \ar[r,simeq,"{\widehat{\epsilon}_e}"'{inner sep=1ex}] \& Q
						\end{cd}
						in $\L$. We then conclude that $e$ is an $\N$-2-cokernel, and precisely the $\N$-2-cokernel of the $\N$-2-kernel of $e$. Moreover, it is straightforward to see that the structure isomorphic 2-cell that exhibits such cokernel is given, up to an invertible null 2-cell, by the structure isomorphic 2-cell of the $\N$-2-kernel of $e$. Notice that the proof of this uses \thex\ref{teorequividealshavesameker}, to build the structure 2-cell of the $\N$-2-cokernel $c_{k_e}$ of $k_e$ from the structure 2-cell given by $c_{k_e}$ being the $\N'$-2-cokernel of $k_e$. Dually, $m$ is the $\N$-2-kernel of the $\N$-2-cokernel of $m$, with structure isomorphic 2-cell similarly good.
						
						We conclude that $f$ factorizes up to isomorphism as an $\N$-2-cokernel followed by an $\N$-2-kernel. We can also show that $\E$ and $\M$ coincide with the classes of all $\N$-2-cokernels and all $\N$-2-kernels respectively. Indeed, we already know that every morphism in $\M$ is an $\N$-2-kernel. Given then an $\N$-2-kernel $t\:T\to A$, say of $f$, then since also $k_e$ is an $\N$-2-kernel of $f$ there exists an equivalence $i\:T\simeq K(e)$ and an isomorphic 2-cell $k_e\c i\iso t$. Whence $t\in \M$ because $k_e\in \M$, by the axioms of factorization system on a 2-category. And dually for $\E$.
						
						Thanks to this, we deduce that every $\N$-2-kernel is an $\N$-2-kernel of its $\N$-2-cokernel, and that every $\N$-2-cokernel is an $\N$-2-cokernel of its $\N$-2-kernel. Indeed, we had already shown that this holds for morphisms in $\M$ and in $\E$ respectively.
						
						It remains to prove that $\N$ is a closed 2-ideal. We show that $\N$-2-kernels $m\:Q\to A$ reflect null morphisms. So consider $s\:D\to Q$ such that $m\c s$ is isomorphic to a null morphism $n\:D\anull{}A$, via an isomorphic 2-cell $\delta$. Then by the universal property of the $\N$-2-cokernel $c_{\id{D}}$ of $\id{D}$, there exist $w\:C(\id{D})\to A$ and an isomorphic $2$-cell as on the left below such that the pasting on the right below is a null 2-cell $\Gamma$:
						\begin{eqD*}
							\begin{cd}*[5.4][5.4]
								D \ar[r,iso,shift right=5ex,"{\gamma}"{inner sep=1ex}]\ar[d,twoheadrightarrow,"{c_{\id{D}}}"']\ar[r,"{s}"]\& Q \ar[d,rightarrowtail,"{m}"]\\
								C(\id{D}) \ar[r,"{w}"']\& A
							\end{cd}\quad\qquad
							\begin{cd}*
								{} \&[-2ex] {} \&[-1ex] {Q}\ar[r,"m",rightarrowtail] \&[-2ex] {Z}\\[-4ex]
								{D} \arrow[rr,"{\hspace{2ex}\beta}"'{pos=0.57, inner sep=0.01ex}, iso, shift right= 1.65ex, start anchor={[xshift=-0.1ex]}] \arrow[rr,"{\hspace{2ex}\delta^{-1}}"{pos=0.53, inner sep=0.9ex}, iso, shift left= 4.2ex, start anchor={[xshift=-0.1ex]}] \arrow[rrru,"", null, bend left=42] \arrow[rr,"", null, bend right=30] \arrow[rrru,"", null, bend right=72] \arrow[r,equal] \& {D} \ar[ru,"s"]\arrow[r,"{\hspace{2ex}\nu}"'{pos=0.99}, iso, shift right= 3.1ex, start anchor={[xshift=3.5ex]}] \arrow[r,"{\hspace{2ex}\gamma}"{pos=0.91}, iso, shift left= 2.1ex, start anchor={[xshift=4.9ex]}] \arrow[r,"{c_{\id{D}}}"{inner sep=0.15ex}] \& {C} \arrow[ru,"{\exists w}"', dashed] \& {}
							\end{cd}
						\end{eqD*}
						By the diagonal lifting property of the factorization system $(\E,\M)$, there exist $d\:C(\id{D})\to Q$ and isomorphic $2$-cells $\sigma\:s\iso d\c c_{\id{D}}$ and $\tau\:m\c d\iso w$ such that
						\begin{eqD*}
							\begin{cd}*[5.4][5.4]
								D \ar[r,iso,shift right=5ex,"{\gamma}"{inner sep=1ex}]\ar[d,twoheadrightarrow,"{c_{\id{D}}}"']\ar[r,"{s}"]\& Q \ar[d,rightarrowtail,"{m}"]\\
								C(\id{D}) \ar[r,"{w}"']\& A
							\end{cd}
							\quad\h[2] = \quad
							\begin{cd}*[5.4][5.4]
								D \ar[r,iso,shift right=3.5ex,"{\sigma}"{inner sep=1ex},start anchor={[xshift=-5ex]}]\ar[d,twoheadrightarrow,"{c_{\id{D}}}"']\ar[r,"{s}"]\& Q \ar[d,rightarrowtail,"{m}"]\\
								C(\id{D})\ar[r,iso,shift left=3.4ex,"{\tau}"{inner sep=1ex},start anchor={[xshift=4.65ex]}]\ar[ru,"d"'{inner sep=0.2ex}] \ar[r,"{w}"']\& A
							\end{cd}
						\end{eqD*}
						Since $c_{\id{D}}$ is isomorphic to a null morphism $y$, say via $\beta$, as it is the $\N$-2-cokernel of an identity, $\sigma$ (pasted with $\beta$ and $\nu$) provides an isomorphic 2-cell $\psi$ between $s$ and a null morphism. It is then straightforward to show that $\psi$ satisfies the required axiom written in \defx\ref{defreflnull}, pasting $\Gamma$ with the invertible null 2-cell given by $\tau\ast y$ and an invertible null 2-cell given by the pseudofunctoriality of $N$. Thus $\N$-2-kernels reflect null morphisms.
						
						We now show that $\N$-2-kernels $m$ reflect null 2-cells. So consider a $2$-cell $\mu\:s\aR{}s'\:D\anull{} Q$ such that $m\ast \mu$ (pasted with appropriate $\nu$'s) is a null 2-cell. We can then produce, as shown earlier, $w,\gamma,\Gamma,d,\sigma,\tau$ starting from $s$ and analogous $w',\gamma',\Gamma',d',\sigma',\tau'$ starting from $s'$. As $s$ and $s'$ are null, we can also choose $\tau$ and $\tau'$ to be identities. Pasting $\Gamma^{-1}$ with $m\ast \mu$ and $\Gamma'$, we obtain a null 2-cell which induces a unique 2-cell $\theta\:w\aR{}w'\:C(\id{D})\to A$ such that $\theta\ast\c_{\id{D}}=\gamma'\c m\ast \mu\c \gamma^{-1}$. But then, by the 2-dimensional diagonal lifting property applied to $\gamma$ and $\gamma'$, the pair of 2-cells $(\mu,\theta)$ induces a 2-cell $\iota\:d\aR{}d'\:C(\id{D})\to Q$ such that $\mu$ coincides with ${\sigma'}^{-1}\c \iota\ast c_{\id{D}}\c \sigma$. It is straightforward to see that the latter is a null 2-cell. We conclude that $\N$-2-kernels reflect null 2-cells. Dually, we obtain that $\N$-2-cokernels coreflect null morphisms and null 2-cells. So $\N$ is a closed 2-ideal.
						
						We prove $(ii)\aR{} (i)$. We define $\E$ and $\M$ to be the classes of all $\N$-2-cokernels and all $\N$-2-kernels respectively. We will denote as $K(e)\aar{k_e}A$ the $\N$-2-kernel of a morphism $e\:A\toe Q$ in $\E$, and as $B\aar{c_m} C(m)$ the $\N$-2-cokernel of a morphism $m\:Q\tom B$ in $\M$. By \prox\ref{propkerarefaith}, every morphism in $\E$ is cofaithful and every morphism in $\M$ is faithful. We prove that $(\E,\M)$ is a factorization system on $\L$, and thus a $(1,1)$-proper one. By assumption, every morphism $f$ in $\L$ factorizes up to isomorphism as a morphism in $\E$ followed by a morphism in $\M$. It is then easy to see that if $m\:Q\to B$ is an $\N$-2-kernel and $i\:R\simeq Q$ is an equivalence, then $m\c i$ is an $\N$-2-kernel, of the same morphism. And also that if $m$ is an $\N$-2-kernel, say of $g$ with structure isomorphic 2-cell $\alpha$, and $\ell$ is isomorphic to $m$, via $\psi$, then $l$ is as well an $\N$-2-kernel of $g$, with structure isomorphic 2-cell $\alpha\c g\ast \psi$. Dual conditions then hold for $\E$ as well.
						
						We prove the orthogonality condition for $(\E,\M)$. So let $e\in \E$ be the $\N$-2-cokernel $D\toe R$ of a morphism $g\:C\to D$ with structure isomorphic 2-cell $\beta$, and $m\in \M$ be the $\N$-2-kernel $Q\tom A$ of a morphism $f\:A\to B$ with structure isomorphic 2-cell $\alpha$. And consider an isomorphic 2-cell
						\begin{cd}[5][5]
							D \ar[r,"s"]\ar[d,"e"',twoheadrightarrow,""{name=A}] \& Q \ar[d,"m",rightarrowtail,""'{name=B}]\\
							R \ar[r,"t"'] \& A
							\ar[from=A,to=B,iso,"\phi"{inner sep=1ex},shift right=0.8ex]
						\end{cd}
						Then $f\c t\c e$ is isomorphic to a null morphism, via
						\begin{cd}[3][6]
							\&[-0.5ex] R \ar[r,"t"]\&[1.5ex] A \ar[r,"f"]\& B \\
							D \ar[rrru,null,bend right=47]\ar[r,"s"']\ar[ru,"e",twoheadrightarrow,""'{name=A}] \& Q \ar[rru,null,bend right=20]\ar[ru,"m"',rightarrowtail,""{name=B}] \& \hphantom{.}
							\ar[from=1-2,to=1-3,iso,"{\phi^{-1}}"{inner sep=0.8,pos=0.65},shift right=3.4ex,start anchor={[xshift=-4.65em]}]
							\ar[from=1-3,to=1-4,iso,"{\alpha}"{inner sep=0.8,pos=0.61},shift right=3.4ex,start anchor={[xshift=-4.65em]}]
							\ar[from=2-2,to=2-3,iso,"{\nu}"{inner sep=0.8,pos=0.63},shift right=3.2ex,start anchor={[xshift=-2em]}]
						\end{cd}
						But $e$ coreflects null morphisms, so there exists an isomorphic 2-cell $\psi\:f\c t\iso \widehat{\psi}$ with $\widehat{\psi}$ a null morphism such that $\nu\c \psi\ast e$ coincides with the pasting above, up to an invertible null 2-cell. $\psi$ then induces $d\:R\to Q$ and $\delta$ as below, by the universal property of the $\N$-2-kernel $m$, making the following pasting into an invertible null 2-cell:
						\begin{cd}[6.7][6.7]
							{R} \arrow[rrrd,""{name=C}, null, bend right=69] \arrow[rd,"{\exists d}"'{name=A}, dashed] \arrow[rrd,"t"{name=B}, bend left=20] \arrow[rrrd,"", null, bend left] \&[-2ex] {} \&[-2ex] {} \&[-2ex] {}\\[-4ex]
							{} \& {Q} \arrow[r,"{\exists \delta}"'{pos=0.65}, iso, shift left= 3ex, start anchor={[xshift=-8ex]}] 
							\arrow[r,"{\nu}"'{pos=0.65}, iso, shift right= 2.8ex, start anchor={[xshift=-8ex]}] 
							\arrow[r,"m"] \arrow[rr,""{name=D}, bend right=42, null] \& {A} \arrow[r,"{\psi^{-1}}"'{pos=0.65}, iso, shift left= 3.8ex, start anchor={[xshift=-9ex]}] \arrow[r,"{\alpha}"'{pos=0.65}, iso, shift right= 2ex, start anchor={[xshift=-8ex]}]  \arrow[r,"f"] \& {B}
						\end{cd}
						It is then straightforward to see that
						\begin{cd}
							D \ar[r,iso,shift right=3.75ex,"\phi"{pos=0.4,inner sep=1ex},start anchor={[xshift=-0.7ex]}]\ar[d,"e"']\ar[r,"s"] \& Q \ar[r,"m"] \& A \\
							R \ar[rr,iso,shift left=3.15ex,"\delta"{pos=0.5,inner sep=1ex},start anchor={[xshift=0ex]}]\ar[r,"d"']\ar[rru,"t"{inner sep=0.25ex},bend left=10] \& Q \ar[ru,"m"'] \& \hphantom{.}
						\end{cd}
						induces an isomorphic 2-cell $\sigma\:s\aR{}d\c e$ such that
						\begin{eqD*}
							\begin{cd}*[5][5]
								D \ar[r,"s"]\ar[d,"e"',twoheadrightarrow,""{name=A}] \& Q \ar[d,"m",rightarrowtail,""'{name=B}]\\
								R \ar[r,"t"'] \& A
								\ar[from=A,to=B,iso,"\phi"{inner sep=1ex},shift right=0.8ex]
							\end{cd}
							\quad\h[2] = \quad
							\begin{cd}*[5.4][5.7]
								D \ar[r,iso,shift right=3.5ex,"{\sigma}"{inner sep=1ex},start anchor={[xshift=-5ex]}]\ar[d,twoheadrightarrow,"{e}"']\ar[r,"{s}"]\& Q \ar[d,rightarrowtail,"{m}"]\\
								R\ar[r,iso,shift left=3ex,"{\delta^{-1}}"{inner sep=1ex},start anchor={[xshift=4.65ex]}]\ar[ru,"d"'{inner sep=0.1ex,pos=0.4}] \ar[r,"{t}"']\& A
							\end{cd}
						\end{eqD*}
						This concludes the 1-dimensional part of the orthogonality condition. For the 2-dimensional part, consider squares $(s,t,\phi)$ and $(s',t',\phi')$, filled by $(d,\sigma,\rho)$ and $(d',\sigma',\rho')$ respectively. Let then $(\lambda,\mu)$ be a 2-cell between such squares. We apply the 2-dimensional universal property of the $\N$-2-kernel $m$ to the 2-cell $\rho'^{-1}\c\mu\c \rho$. It is straightforward to see that this 2-cell gives indeed a null 2-cell, thanks to the fact that $e$ coreflects null 2-cells. It is then induced a 2-cell $\iota\:d\aR{}d'$ such that
						\begin{eqD*}
							\begin{cd}*[6][9]
								\& Q \ar[d,"m",rightarrowtail] \\
								R \ar[r,"t'"']\ar[ru,bend left=17,"d"{inner sep=0.1ex},""{name=A}] \ar[ru,bend right=17,"d'"'{inner sep=0.1ex},""'{name=B}] 
								\ar[r,iso,shift left=2.7ex,"\rho'"'{pos=0.6},start anchor={[xshift=7ex]}]
								\& A
								\ar[from=A,to=B,Rightarrow,"\iota",shorten <=0.5ex,shorten >=0.5ex]
							\end{cd}\quad \h[2] = \quad \h[-3]
							\begin{cd}*[6][9]
								\& Q \ar[d,"m",rightarrowtail] \\
								R \ar[r,"t",bend left=16,""{name=A}]\ar[r,"t'"',bend right=20,""'{name=B}]\ar[ru,bend left=17,"d"{inner sep=0.1ex},""]
								\ar[r,iso,shift left=3.7ex,"\rho"{pos=0.6,inner sep=1ex},start anchor={[xshift=4.5ex]}]
								\& A
								\ar[from=A,to=B,Rightarrow,"\mu",shorten <=0.5ex,shorten >=0.5ex]
							\end{cd}
						\end{eqD*}
						And the other axiom required to $\iota$ holds because $m$ is faithful. We conclude that $(\E,\M)$ is a factorization system on the 2-category $\L$.
						
						We define $K\:\Esq\to \Msq$ as follows. Given $e\:A\toe Q$, we send it to its $\N$-2-kernel $K(E)\aar{k_e} A$. Given a morphism
						\begin{cd}[5][5]
							A \ar[d,twoheadrightarrow,"e"'{name=A}]\ar[r,"a"] \& A' \ar[d,twoheadrightarrow,"e'"{name=B}] \\
							Q \ar[r,"q"'] \& Q'
							\ar[from=A,to=B,iso,shift right=0.5ex,"\lambda"{inner sep=1ex}]
						\end{cd}
						in $\Esq$, we send it to the square given by $(\widehat{K(\lambda)},a,K(\lambda))$ induced as below by the $\N$-2-kernel $k_{e'}$, making the following pasting into an invertible null 2-cell:
						\begin{cd}
							{K(e)} \ar[rr,iso,shift left=2ex,"{\alpha^{-1}}"{inner sep=1ex}]\arrow[rrrd,""{name=C}, null, bend right=60]\arrow[rrrd, null, bend left=69] \arrow[rd,"{\exists \widehat{K(\lambda)}}"'{name=A,inner sep=0.05ex}, dashed] \arrow[r,rightarrowtail,"k_e"{name=B}] \arrow[rr,"", null, bend left=40] \&[-2ex] {A}\ar[rr,iso,bend left=6ex,"\nu"'{pos=0.65}]\ar[rd,"a"{inner sep=0.15ex}] \ar[r,"e"]\&[-2ex] {Q} \ar[rd,"q"]\&[-2ex] {\hphantom{.}}\\[-4ex]
							{} \& {K(e')} \arrow[r,"{\exists K(\lambda)}"'{pos=0.65}, iso, shift left= 3.75ex, start anchor={[xshift=-14ex]}] 
							\arrow[r,"{\nu}"'{pos=0.59}, iso, shift right= 2.8ex, start anchor={[xshift=-13ex]}] 
							\arrow[r,rightarrowtail,"{k_{e'}}"] \arrow[rr,""{name=D}, bend right=40, null] \& {A'} \arrow[r,"{\lambda^{-1}}"'{pos=0.65}, iso, shift left= 3.3ex, start anchor={[xshift=-8ex]}] \arrow[r,"{\alpha'}"'{pos=0.6}, iso, shift right= 2ex, start anchor={[xshift=-8ex]}]  \arrow[r,"e'"] \& {Q'}
						\end{cd}
						Given a 2-cell $(\xi,\zeta)\:(a,q,\lambda)\aR{}(a',q',\lambda')\:e\to e'$ in $\Esq$, we send it to the 2-cell $([K(\xi,\zeta)],\xi)$, where $[K(\xi,\zeta)]\:\widehat{K(\lambda)}\aR{}\widehat{K(\lambda')}$ is the unique 2-cell induced by the 2-dimensional universal property of the $\N$-2-kernel $k_{e'}$ starting from the 2-cell $K(\lambda')\c \xi\ast k_e\c K(\lambda)^{-1}$. It is straightforward to show that $K$ is a normal pseudofunctor over $\L$. Dually, we have that $C\:\Msq \to \Esq$ defined by taking $\N$-2-cokernels is a normal pseudofunctor over $\L$.
						
						We prove that $K$ and $C$ give a biequivalence over $\L$. It suffices to construct pseudonatural transformations $\eta\:\Id{}\aR{}K\c C$ and $\epsilon\:C\c K\aR{}\Id{}$ over $\L$ such that the components of $\eta$ and $\eps$ are equivalences. Consider $m\:Q\tom B$. By assumption, $m$ is an $\N$-2-kernel of its $\N$-2-cokernel, with structure isomorphic 2-cell given, up to invertible null 2-cell, by the structure isomorphic 2-cell of the $\N$-2-cokernel of $m$. We define the component $\eta_m$ of $\eta$ on $m$ as the square (actually, triangle) $(\widehat{\eta_m},\id{},\eta_m)$ induced by the universal property of the $\N$-2-kernel $k_{c_m}$, making the following pasting into an invertible null 2-cell:
						\begin{cd}
							{Q} \arrow[rrrd,""{name=C}, null, bend right=69] \arrow[rd,"{\exists \widehat{\eta_m}}"'{name=A,inner sep=-0.05ex}, dashed] \arrow[rrd,rightarrowtail,"m"{name=B, inner sep=0.25ex}, bend left=20] \arrow[rrrd,"", null, bend left=27] \&[-2ex] {} \&[0ex] {} \&[-2ex] {}\\[-3ex]
							{} \& {K} \arrow[r,"{\exists \eta_m^{-1}}"{pos=0.36,inner sep=0.5ex}, iso, shift left= 3ex, start anchor={[xshift=-8ex]}] 
							\arrow[r,"{\nu}"'{pos=0.63}, iso, shift right= 2.8ex, start anchor={[xshift=-8ex]}] 
							\arrow[r,"k_{c_m}",rightarrowtail] \arrow[rr,""{name=D}, bend right=40, null] \& {B} \arrow[r,"{\beta^{-1}}"'{pos=0.65}, iso, shift left= 3.7ex, start anchor={[xshift=-8ex]}] \arrow[r,"{\alpha}"'{pos=0.65}, iso, shift right= 2.2ex, start anchor={[xshift=-8ex]}]  \arrow[r,"{c_m}"] \& {C(m)}
						\end{cd}
						where $\beta$ is given by the cokernel $c_m$ of $m$ and $\alpha$ is given by the kernel $k_{c_m}$ of $c_m$. Since both $m$ and $k_{c_m}$ are $\N$-2-kernels of $c_m$, we deduce that $\widehat{\eta_{m}}$ is an equivalence. It is straightforward to see that this implies that the morphism $\eta_m$ in $\Msq$ is an equivalence in $\Msq$. Given a morphism
						\begin{cd}[5][5]
							Q \ar[d,rightarrowtail,"m"'{name=A}]\ar[r,"q"] \& Q' \ar[d,rightarrowtail,"m'"{name=B}] \\
							B \ar[r,"b"'] \& B'
							\ar[from=A,to=B,iso,shift right=0.5ex,"\lambda"{inner sep=1ex}]
						\end{cd}
						in $\Msq$, we define the structure 2-cell $\eta_\lambda$ of $\eta$ on it as the isomorphic 2-cell $([\eta_{\lambda}],\id{})$ in $\Msq$ where $[\eta_{\lambda}]\:\widehat{K(C(\lambda))}\c \widehat{\eta_m}\aR{}\widehat{\eta_{m'}}\c q$ is the unique isomorphic 2-cell such that $k_{c_{m'}}\ast [\eta_{\lambda}]$ is equal to
						\[\begin{tikzcd}[row sep=3ex,column sep=4ex]
							& {K(c_m)} && {K(c_{m'})} \\
							Q && B && {B'} \\
							& {Q'} && {K(c_{m'})} & \hphantom{.}
							\arrow["{\widehat{K(C(\lambda))}}", from=1-2, to=1-4]
							\arrow[""{name=0, anchor=center, inner sep=0}, "{k_{c_m}}"{description}, from=1-2, to=2-3]
							\arrow[""{name=1, anchor=center, inner sep=0}, "{k_{c_{m'}}}", from=1-4, to=2-5]
							\arrow[""{name=2, anchor=center, inner sep=0}, "{\widehat{\eta_m}}", from=2-1, to=1-2]
							\arrow["m"', from=2-1, to=2-3]
							\arrow[""{name=3, anchor=center, inner sep=0}, "q"', from=2-1, to=3-2]
							\arrow["b"{description}, from=2-3, to=2-5]
							\arrow[""{name=4, anchor=center, inner sep=0}, "{m'}"{description}, from=3-2, to=2-5]
							\arrow["{\widehat{\eta_{m'}}}"', from=3-2, to=3-4]
							\arrow["{k_{c_{m'}}}"', from=3-4, to=2-5,bend right=20, shorten <=-0.5ex]
							\arrow["{K(C(\lambda))}"{inner sep=1ex},shift right=1ex,iso, from=0, to=1]
							\arrow["{\eta_m}"{inner sep=1ex}, shift right=1ex,iso, from=2, to=0]
							\arrow["{\lambda^{-1}}"{inner sep=1ex}, iso, from=3, to=4]
							\ar[from=3-2,to=3-5,"{\eta_{m'}^{-1}}"{inner sep=1ex,pos=0.6},iso, shift left=2ex,start anchor={[xshift=6em]}]
						\end{tikzcd}\]
						It is straightforward to see that such an isomorphic 2-cell $[\eta_{\lambda}]$ can be induced by the 2-dimensional universal property of the $\N$-2-kernel $k_{c_{m'}}$, by construction of $\eta_m$, $\eta_{m'}$, $C(\lambda)$ and $K(C(\lambda))$. And it is straightforward to prove that $\eta$ is a pseudonatural transformation over $\L$, thanks to the uniqueness of $[\eta_{\lambda}]$ and the fact that $\N$-2-kernels are faithful. Dually, we have a pseudonatural equivalence $\epsilon\:C\c K\aR{}\Id{}$ over $\L$. And we conclude that $K$ and $C$ give a biequivalence over $\L$.
					\end{proof}
					
					\begin{corollary}\label{corollpuppe}
						Let $\L$ be a 2-pointed 2-category, with a fixed choice $0$ of a representative for a bizero object. The following conditions are equivalent:
						\begin{enumT}
							\item $\L$ satisfies condition $(i)$ of \thex\ref{theorchargrandistwocat} with $K$ and $C$ sending identities to morphisms from $0$ and into $0$ respectively;
							\item $\L$ satisfies condition $(ii)$ of \thex\ref{theorchargrandistwocat} with $\N$ given by the 2-ideal $(0)$ canonically associated to $\L$.
						\end{enumT}
					\end{corollary}
					\begin{proof}
						We prove $(i)\aR{} (ii)$. The 2-ideal $\N$ constructed in the proof of \thex\ref{theorchargrandistwocat} is contained in the 2-ideal $(0)$. It is straightforward to see that $\N$ is equivalent to $(0)$, in the sense of \defx\ref{defeq2ideals}, thanks to the universal property of the bizero object. Whence 2-kernels and 2-cokernels for the two 2-ideals are the same, by \thex\ref{teorequividealshavesameker}. Finally, it is straightforward to show that $\L$ satisfies condition $(ii)$ of \thex\ref{theorchargrandistwocat} with respect to the 2-ideal $(0)$, knowing that it satisfies it with respect to the 2-ideal $\N$.
						
						We prove $(ii)\aR{}(i)$. It is easy to see that $0$ gives $(0)$-2-kernels and $(0)$-2-cokernels of identities. We can thus choose in the construction of $K$ and $C$ of \thex\ref{theorchargrandistwocat} to send identities to morphisms from $0$ and into $0$ respectively.
					\end{proof}
					
					Since the 1-dimensional analogue of \thex\ref{theorchargrandistwocat} characterizes Grandis exact categories, we propose the following definition. Recall that the pointed version of Grandis exact categories is known as Puppe exact category.
					
					\begin{definition}\label{def2exactness}
						We call a 2-category $\L$ \dfn{Grandis 2-exact} if it satisfies one of the two equivalent conditions of \thex\ref{theorchargrandistwocat}. A 2-pointed 2-category $\L$ is \dfn{Puppe 2-exact} if it satisfies one of the two equivalent conditions of \corx\ref{corollpuppe}.
					\end{definition}
					
					
					In the next section we will compare our notion of Puppe exactness in dimension 2 with the existing literature.
					
					\begin{corollary}\label{corollckkc}
						In a Grandis 2-exact (or Puppe 2-exact) 2-category $\L$, every morphism $f$ factorizes up to isomorphism as the $\N$-2-cokernel of its $\N$-2-kernel followed by the $\N$-2-kernel of its $\N$-2-cokernel:
						\begin{cd}[0]
							A \ar[rr,iso,shift right=0.75 ex]\ar[rr,bend left=10,"f"]\ar[rd,twoheadrightarrow,"{c_{k_f}}"'{pos=0.7}]\&\& B \\
							\& Q \ar[ru,rightarrowtail,"{k_{c_f}}"'{pos=0.3}]
						\end{cd}
					\end{corollary}
					\begin{proof}
						By definition of Grandis 2-exact 2-category, every morphism $f$ factorizes as an $\N$-2-cokernel $e$ followed by an $\N$-2-kernel $m$. Since every $\N$-2-cokernel is an $\N$-2-cokernel of its $\N$-2-kernel, we obtain that $e=c_{k_e}$; and dually $m=k_{c_m}$. By the proof of \thex\ref{theorchargrandistwocat}, the $\N$-2-kernel $k_e$ of $e$ coincides with the $\N$-2-kernel of $f$, and the $\N$-2-cokernel $c_m$ of $m$ coincides with the $\N$-2-cokernel of $f$. Whence
						$$f=m\c e=k_{c_m}\c c_{k_e}=k_{c_f}\c c_{k_f}.$$
						Analogously for a Puppe 2-exact (2-pointed) $\L$.
					\end{proof}
					
					In fact, we can prove the following.
					
					\begin{proposition}\label{propfactthreepieces}
						Let $\L$ be a 2-category equipped with a 2-ideal $\N$ of null morphisms and null 2-cells, such that $\L$ has all $\N$-2-kernels and $\N$-2-cokernels and $\N$ is a closed 2-ideal. Then every morphism $f$ in $\L$ has a three-pieces factorization up to isomorphism as
						\begin{cd}[5.25][5.25]
							A \arrow[r,"{f}"] \arrow[d,two heads,"{c_{k_f}}"',""{name=A}] \& B \\
							C(k_f) \arrow[r,"z"'] \& K(c_f) \arrow[u,rightarrowtail,"{k_{c_f}}"',,""{name=B}]
							\arrow[from=A,to=B, iso]
						\end{cd}
						where $c_{k_f}$ is the $\N$-2-cokernel of the $\N$-2-kernel of $f$ and $k_{c_f}$ is the $\N$-2-kernel of the $\N$-2-cokernel of $f$.
					\end{proposition}
					\begin{proof}
						By the universal property of the $\N$-2-cokernel $c_{k_f}$ (with structure 2-cell $\delta$), there exist $u\:C(k_f)\to B$ and $\mu\:f\aR{}u\c c_{k_f}$ such that the pasting
						\begin{cd}
							{} \&[-2ex] {} \&[-2ex] {} \&[-2ex] {B}\\[-4ex]
							{K(f)} \arrow[rr,"{\hspace{2ex}\delta}"'{pos=0.57, inner sep=0.01ex}, iso, shift right= 1.65ex, start anchor={[xshift=-0.1ex]}] \arrow[rr,"{\hspace{2ex}\alpha^{-1}}"'{pos=0.57, inner sep=0.001ex}, iso, shift left= 4.5ex, start anchor={[xshift=-0.1ex]}] \arrow[rrru,"", null, bend left=40] \arrow[rr,"", null, bend right=30] \arrow[rrru,"", null, bend right=72] \arrow[r,rightarrowtail,"{k_f}"] \& {A} \arrow[r,"{\hspace{2ex}\nu}"'{pos=0.99}, iso, shift right= 4.25ex, start anchor={[xshift=3.5ex]}] \arrow[r,"{\hspace{2ex}\mu}"'{pos=0.69,inner sep=0.1ex}, iso, shift left= 2.8ex, xshift=3.9ex] \arrow[rru,"f", bend left=21]\arrow[r,two heads,"{c_{k_f}}"{inner sep=0.2ex}] \& {C(k_f)} \arrow[ru,"{\exists u}"', dashed] \& {}
						\end{cd}
						is an invertible null 2-cell, where $\alpha$ is the structure 2-cell of the $\N$-2-kernel $k_f$ of $f$. Then $c_f\c u\c c_{k_f}$ is isomorphic to null, pasting $\mu^{-1}$ with the structure 2-cell $\beta$ of the $\N$-2-cokernel $c_f$. Since $\N$ is a closed 2-ideal, the $\N$-2-cokernel $c_{k_f}$ coreflects null morphisms, and thus $c_f\c u$ is isomorphic to a null morphism via a nice invertible 2-cell $\psi$. We can then apply the universal property of the $\N$-2-kernel $k_{c_f}$ (with structure 2-cell $\gamma$) to induce $z\:C(k_f)\to K(c_f)$ and $\xi\:u\aR{}k_{c_f}\c z$ such that the pasting
						\begin{cd}
							{C(k_f)} \arrow[rrrd,""{name=C}, null, bend right=69] \arrow[rd,"{\exists z}"'{name=A}, dashed] \arrow[rrd,"u"{name=B}, bend left=20] \arrow[rrrd,"", null, bend left] \&[-2ex] {} \&[-2ex] {} \&[-2ex] {}\\[-4ex]
							{} \& {K(c_f)} \arrow[r,"{\exists \xi}"'{pos=0.65}, iso, shift left= 3ex, start anchor={[xshift=-8ex]}] 
							\arrow[r,"{\nu}"'{pos=0.7}, iso, shift right= 2.8ex, xshift=-9.3ex] 
							\arrow[r,rightarrowtail,"{k_{c_f}}"'{inner sep=0.17ex}] \arrow[rr,""{name=D}, bend right=42, null] \& {B} \arrow[r,"{\psi^{-1}}"'{pos=0.65}, iso, shift left= 4ex, start anchor={[xshift=-8ex]}] \arrow[r,"{\gamma}"'{pos=0.65}, iso, shift right= 2ex, start anchor={[xshift=-8ex]}]  \arrow[r,two heads,"{c_f}"] \& {C(f)}
						\end{cd}
						is an invertible null 2-cell. We have thus obtained a three-pieces factorization of $f$ as 
						\begin{cd}[5.75][5.75]
							A \arrow[r,"{f}"]\arrow[r,iso,"{\mu}"'{pos=0.33},shift right=2.5ex,xshift=-2.7ex] \arrow[d,two heads,"{c_{k_f}}"',""{name=A}] \& B \\
							C(k_f) \arrow[r,iso,"{\xi}"{pos=0.67},shift left=2.5ex,xshift=2.7ex]\arrow[ru,"u"{inner sep=0.2ex}]\arrow[r,"z"'] \& K(c_f) \arrow[u,rightarrowtail,"{k_{c_f}}"',,""{name=B}]
						\end{cd}
					\end{proof}
					
					\begin{remark}\label{remdualthreepiecesfact}
						Dually to the proof of \prox\ref{propfactthreepieces}, we also obtain a three-pieces factorization of $f$ as
						\begin{cd}[5.75][5.75]
							A \arrow[rd,"v"{inner sep=0.2ex}]\arrow[r,"{f}"]\arrow[r,iso,"{\lambda}"'{pos=0.67},shift right=2.5ex,xshift=2.7ex] \arrow[d,two heads,"{c_{k_f}}"',""{name=A}] \& B \\
							C(k_f) \arrow[r,iso,"{\xi'}"{pos=0.33},shift left=2.5ex,xshift=-2.7ex]\arrow[r,"z'"'] \& K(c_f) \arrow[u,rightarrowtail,"{k_{c_f}}"',,""{name=B}]
						\end{cd}
						applying first the universal property of the $\N$-2-kernel $k_{c_f}$ and then the universal property of the $\N$-2-cokernel $c_{k_f}$.
						
						It is straightforward to prove that there exists an isomorphic 2-cell $\eta\:z'\aR{}z$ such that
						\begin{eqD*}
							\begin{cd}*[6.75][6.75]
								A\arrow[rd,"{f}"'{description}] \arrow[r,"{v}"]\arrow[r,iso,"{\lambda^{-1}}"'{pos=0.67},shift right=2.5ex,xshift=2.9ex] \arrow[d,two heads,"{c_{k_f}}"',""{name=A}] \& K(c_f)\arrow[d,rightarrowtail,"{k_{c_f}}",,""{name=B}] \\
								C(k_f) \arrow[r,iso,"{\mu}"{pos=0.33,inner sep=0.6ex},shift left=2.5ex,xshift=-2.9ex]\arrow[r,"u"'] \& B
							\end{cd}
							\quad \h[1.5] = \h[1.5] \quad
							\begin{cd}*[6.75][6.75]
								A \arrow[r,"{v}"]\arrow[r,iso,"{\xi'}"'{pos=0.33},shift right=2.3ex,xshift=-2.9ex] \arrow[d,two heads,"{c_{k_f}}"',""{name=A}] \& K(c_f)\arrow[d,rightarrowtail,"{k_{c_f}}",,""{name=B}] \\
								C(k_f) \arrow[r,iso,"{\xi^{-1}}"{pos=0.67,inner sep=0.6ex},shift left=2.3ex,xshift=2.9ex]\arrow[ru,bend left=15,"z'"{inner sep=0.02ex},""'{name=A}]\arrow[ru,bend right=15,"z"'{pos=0.4,inner sep=0.12ex},""{name=B}]\arrow[r,"u"'] \& B
								\arrow[r,iso,"{\eta}",from=A,to=B]
							\end{cd}
						\end{eqD*}
						using that $k_{c_f}$ reflects null 2-cells since $\N$ is a closed 2-ideal.
					\end{remark}
					
					\begin{remark}
						In \cite{Dup08}, Dupont described the problem of finding a three-pieces factorization of morphisms in dimension 2, using 2-cokernels and 2-kernels. Here we presented a solution to this problem, adding the assumption that the 2-ideal is closed. Notice that in the proof of \prox\ref{propfactthreepieces} we needed that $c_{k_f}$ coreflects null morphisms rather than just, for example, weakly coreflects null morphisms, since in general 
						\begin{cd}[5][5]
							{K(f)} \& A \& B \& {C(f)}
							\arrow[rightarrowtail,"{k_f}"', from=1-1, to=1-2]
							\arrow[null,bend left=35, from=1-1, to=1-3]
							\arrow[iso,"{\alpha^{-1}}"{pos=0.34, inner sep=0ex},shift left=4, from=1-1, to=1-3]
							\arrow[null,bend left=43, from=1-1, to=1-4]
							\arrow[null,bend right=43, from=1-1, to=1-4]
							\arrow[iso,"\nu"{pos=0.56}, shift left=9, from=1-1, to=1-4]
							\arrow[iso,"\nu"{pos=0.44}, shift right=9, from=1-1, to=1-4]
							\arrow["f", from=1-2, to=1-3]
							\arrow[null,bend right=35, from=1-2, to=1-4] \arrow[iso,"\beta"'{pos=0.6}, shift right=4, from=1-2, to=1-4]
							\arrow[two heads,"{c_f}", from=1-3, to=1-4]
						\end{cd}
						need not be an invertible null 2-cell, where $\alpha$ and $\beta$ are the structure 2-cells of $k_f$ and $c_f$ respectively. The relation between the closedness condition of a 2-ideal and Dupont's work is explored in next section.
					\end{remark}
					
					\begin{corollary}[of \thex\ref{theorchargrandistwocat}]\label{corollfactingrandis2exact}
						In a Grandis 2-exact (or Puppe 2-exact) 2-category $\L$, the morphisms $z$ of the three-pieces factorizations of morphisms described in \prox\ref{propfactthreepieces} are equivalences, so that every morphism $f$ factorizes up to isomorphism as the $\N$-2-cokernel of its $\N$-2-kernel followed by the $\N$-2-kernel of its $\N$-2-cokernel.
					\end{corollary}
					\begin{proof}
						We already proved in \corx\ref{corollckkc} that every morphism $f$ factorizes as 
						\begin{cd}[0]
							A \ar[rr,iso,shift right=0.75 ex]\ar[rr,bend left=10,"f"]\ar[rd,twoheadrightarrow,"{c_{k_f}}"'{pos=0.7}]\&\& B \\
							\& Q \ar[ru,rightarrowtail,"{k_{c_f}}"'{pos=0.3}]
						\end{cd}
						It is straightforward to prove that $k_{c_f}\:Q\to B$ coincides (up to isomorphism) with the induced morphism $u$ constructed in the proof of \prox\ref{propfactthreepieces}, so that $u$ is an $\N$-2-kernel. Dually, $c_{k_f}\:A\to Q$ coincides (up to isomorphism) with the induced morphism $v$ described in \remx\ref{remdualthreepiecesfact}, and $v$ is thus an $\N$-2-cokernel. By \thex\ref{theorchargrandistwocat}, $\L$ has a factorization system given by all $\N$-2-cokernels and all $\N$-2-kernels. The equation 
						\begin{eqD*}
							\begin{cd}*[6.75][6.75]
								A\arrow[rd,"{f}"'{description}] \arrow[r,two heads,"{v}"]\arrow[r,iso,"{\lambda^{-1}}"'{pos=0.67},shift right=2.5ex,xshift=2.9ex] \arrow[d,two heads,"{c_{k_f}}"',""{name=A}] \& K(c_f)\arrow[d,rightarrowtail,"{k_{c_f}}",,""{name=B}] \\
								C(k_f) \arrow[r,iso,"{\mu}"{pos=0.33,inner sep=0.6ex},shift left=2.5ex,xshift=-2.9ex]\arrow[r,rightarrowtail,"u"'] \& B
							\end{cd}
							\quad \h[1.5] = \h[1.5] \quad
							\begin{cd}*[6.75][6.75]
								A \arrow[r,two heads,"{v}"]\arrow[r,iso,"{\xi'}"'{pos=0.33},shift right=2.3ex,xshift=-2.9ex] \arrow[d,two heads,"{c_{k_f}}"',""{name=A}] \& K(c_f)\arrow[d,rightarrowtail,"{k_{c_f}}",,""{name=B}] \\
								C(k_f) \arrow[r,iso,"{\xi^{-1}}"{pos=0.67,inner sep=0.6ex},shift left=2.3ex,xshift=2.9ex]\arrow[ru,bend left=15,"z'"{inner sep=0.02ex},""'{name=A}]\arrow[ru,bend right=15,"z"'{pos=0.4,inner sep=0.12ex},""{name=B}]\arrow[r,rightarrowtail,"u"'] \& B
								\arrow[r,iso,"{\eta}",from=A,to=B]
							\end{cd}
						\end{eqD*}
						of \remx\ref{remdualthreepiecesfact} then shows that $z$ is the lifting given by such factorization system that connects two different factorizations of $f$ (up to iso) as an $\N$-2-cokernel followed by an $\N$-2-kernel. And we conclude that $z$ is an equivalence.
					\end{proof}
					
					\begin{remark}
						We can view \corx\ref{corollckkc} and \corx\ref{corollfactingrandis2exact} as saying that a Grandis 2-exact (or Puppe 2-exact) 2-category satisfies a 2-categorical generalization of the first isomorphism theorem from algebra.
					\end{remark}
					
					We conclude this section by showing that our theory captures as examples all Grandis exact categories in the 1-dimensional sense, and thus all abelian categories.
					
					\begin{proposition}
						A locally discrete 2-category $\L$ is Grandis 2-exact (respectively, Puppe 2-exact) if and only if its underlying 1-category is Grandis exact (in the usual 1-dimensional sense).
					\end{proposition}
					\begin{proof}
						In a locally discrete 2-category, all 2-cells are identities and thus all pasting diagrams are commutative diagrams. A 2-ideal of null morphisms and null 2-cells is precisely an ideal of null morphisms on the underlying 1-category of $\L$. Moreover, 2-kernels and 2-cokernels are precisely kernels and cokernels in the usual 1-dimensional sense. Indeed, the 1-dimensional universal property of a 2-kernel gives the existence condition and the 2-dimensional universal property gives the uniqueness condition of the usual universal property of kernels. Reflecting null morphisms in the 2-dimensional sense means precisely reflecting null morphisms in the ordinary sense; while all morphisms automatically reflect null 2-cells. So closed 2-ideals on $\L$ are precisely closed ideals in the ordinary sense. Notice that here there is no difference between weakly reflecting null morphisms and reflecting null morphisms. It is now straightforward to see that $\L$ satisfies condition $(ii)$ of \thex\ref{theorchargrandistwocat} precisely when its underlying 1-category is Grandis exact. Notice that every kernel being the kernel of its cokernel, as well as the dual condition, is automatic in dimension 1.
						
						For the Puppe exact case, notice that a bizero object in $\L$ is precisely a zero object in the underlying 1-category of $\L$. Moreover, the canonical 2-ideal $(0)$ associated to a 2-pointed locally discrete 2-category is exactly the canonical ideal $(0)$ associated to the underlying 1-category (which is pointed). So we conclude. 
					\end{proof}
					
					We then have the following corollary.
					
					\begin{corollary}
						Every Grandis exact category in the ordinary sense is a Grandis 2-exact locally discrete 2-category.
						
						Every Puppe exact category in the ordinary sense, and thus every abelian category, is a Puppe 2-exact locally discrete 2-category.
					\end{corollary}
					
					\begin{remark}
						For a locally discrete 2-category $\L$, condition $(i)$ of \thex\ref{theorchargrandistwocat} precisely means that there exists an orthogonal factorization system $(\E,\M)$ on the underlying 1-category of $\L$ and an equivalence over $\L$ (which becomes an isomorphism when taking isomorphism classes) between the fibration of quotients relative to $(\E,\M)$ and the opfibration of subobjects relative to $(\E,\M)$ (with $K$ and $C$ functors). Indeed, the 1-dimensional orthogonality condition gives the existence of liftings and the 2-dimensional orthogonality condition gives uniqueness of liftings. The pseudo arrow category of $\L$ is just the usual arrow category of the underlying 1-category of $\L$, seen as a locally discrete 2-category. And a biequivalence over $\L$ is the same as an equivalence over the underlying 1-category of $\L$.
						
						Notice that the condition of having a factorization system which is $(1,1)$-proper is void for a locally discrete 2-category $\L$.
						
						Our proof of \thex\ref{theorchargrandistwocat} gives then a new proof for the 1-dimensional result of the second author and Weighill in \cite{JanWei16}, and shows that the properness condition that they ask is redundant.
					\end{remark}

					\section{Comparison with the literature and applications}
					
					In this section, we compare our 2-dimensional notion of Puppe exactness (\defx\ref{def2exactness}) with the notions developed by Dupont (\cite{Dup08}) and Nakaoka (\cite{Nak08,Nak10}) in the pointed case. To do so, we introduce a slightly weaker notion of 2-dimensional Puppe exactness in a natural way and prove that this generalizes most of the theory and the examples developed in \cite{Dup08,Nak08,Nak10}. Furthermore, we show that such weaker notion of Puppe exactness still implies a biequivalence over $\L$ given by taking kernels and cokernels. So we deduce that many examples in the literature enjoy this fundamental exactness property. 
					
					We start by recalling the 2-dimensional Puppe exactness properties studied by Dupont and Nakaoka, in the particular case of (strictly described) categories enriched over pointed groupoids. We refer the reader to the references written below for the notions of $0$-monomorphism, $0$-faithful, fully $0$-faithful, strictly described category enriched over $\mathrm{Gpd}^\ast$ (that is, over pointed groupoids), locally SCG category, root, pip and dual notions. Following the convention 3.2 of \cite{Nak10}, we consider locally SCG categories without requiring condition (LS3+).
					
					\begin{proposition}
						Let $\L$ be a strictly described category enriched over $\mathrm{Gpd}^\ast$ with a zero object (see \cite{Dup08}). Then for a morphism $f$ in $\L$ the following are equivalent:
						\begin{enumT}
							\item $f$ is 0-faithful;
							\item $f$ reflects null 2-cells.
						\end{enumT}
						Moreover, also the following are equivalent:
						\begin{enumT}
							\item $f$ is fully 0-faithful;
							\item $f$ reflects null morphisms and null 2-cells.
						\end{enumT}
					\end{proposition}
					\begin{proof}
						Since $\L$ is strictly described, all morphisms through the zero object are the special zero morphism $0$ between the right source and target. The unique null 2-cell between parallel null morphisms is then the identity. It is now straightforward to conclude from \cite[{\defx{78}, \defx{80}}]{Dup08}.
					\end{proof}
					
					\begin{definition}[{\cite[{\defx{165}}]{Dup08}}] \label{defdupabelian}
						An \emph{abelian $\mathrm{Gpd}^\ast$-category} is a category $\L$ enriched over $\mathrm{Gpd}^\ast$ with a zero object, finite products and coproducts such that
						\begin{enum}
							\item $\L$ has all 2-kernels and 2-cokernels (see \remx\ref{remcomparisonkernels});
							\item every 0-monomorphism is a 2-kernel of its 2-cokernel (with appropriate structure 2-cell), and dually every 0-epimorphism is a 2-cokernel of its 2-kernel;
							\item fully 0-faithful arrows and 0-monomorphisms are stable under pushout, and dually fully 0-cofaithful arrows and 0-epimorphisms are stable under pullback.
						\end{enum}
					\end{definition}
					
					\begin{definition}[{\cite[\defx{3.7}]{Nak08}}]
						A \emph{relatively exact $\mathrm{Gpd}^\ast$-category} is a locally SCG category $\L$ (where SCG stands for symmetric categorical group) such that:
						\begin{enum}
							\item $\L$ has all 2-kernels and 2-cokernels;
							\item every faithful morphism is a 2-kernel of its 2-cokernel (with appropriate structure 2-cell), and dually every cofaithful morphism is a 2-cokernel of its 2-kernel.
						\end{enum}
					\end{definition}
					
					\begin{definition}[{\cite[\prox{179}, \defx{183}]{Dup08}}]
						A \emph{2-Puppe-exact $\mathrm{Gpd}^\ast$-category} (\'a la Dupont) is a category $\L$ enriched over $\mathrm{Gpd}^\ast$ with a zero object such that
						\begin{enum}
							\item $\L$ has all 2-kernels and 2-cokernels;
							\item every 0-faithful arrow is a 2-kernel of its 2-cokernel (with appropriate structure 2-cell), and dually every 0-cofaithful arrow is a 2-cokernel of its 2-kernel;
							\item every fully 0-faithful arrow is a root of its copip (canonically), and dually every fully 0-cofaithful arrow is a coroot of its pip.
						\end{enum}
						
						A \emph{2-abelian $\mathrm{Gpd}^\ast$-category} is a 2-Puppe-exact $\mathrm{Gpd}^\ast$-category (\'a la Dupont) that has all finite products and coproducts.
					\end{definition}
					
					\begin{remark}\label{Nakaokacomparison}
						The above notions studied by Dupont and Nakaoka are compared in \cite{Nak10}. It is shown there that
						\begin{enum}
							\item every strictly described 2-abelian $\mathrm{Gpd}^\ast$-category is a relatively exact $\mathrm{Gpd}^\ast$-category;
							\item every relatively exact $\mathrm{Gpd}^\ast$-category is an abelian $\mathrm{Gpd}^\ast$-category.
						\end{enum}
					\end{remark}
					
					In order to compare our 2-dimensional notion of Puppe exactness with the above definitions, we introduce a slightly weaker version of \defx\ref{def2exactness}. In particular, we need to weaken the concept of closed ideal, using the weaker form of reflection of null morphisms that we introduced in \defx\ref{defweaklyreflect}. Recall also \prox\ref{charactweaklyreflect}.
					
					\begin{definition}
						Let $\L$ be a 2-category and $\N=(N,\nu)$ be a 2-ideal of null morphisms and null 2-cells in $\L$. Assume that $\L$ has all $\N$-2-kernels and $\N$-2-cokernels. We call $\N$ a \emph{weakly closed 2-ideal} if all $\N$-2-kernels weakly reflect null morphisms and reflect null 2-cells, and all $\N$-2-cokernels weakly coreflect null morphisms and coreflect null 2-cells.
					\end{definition}
					
					\begin{remark}
						Clearly every closed 2-ideal is weakly closed.
						
						Recall from \prox\ref{charactweaklyreflect} that all $\N$-2-kernels weakly reflecting null morphisms is equivalent to all $\N$-2-cokernels weakly coreflecting null morphisms.
					\end{remark}
					
					\begin{proposition}\label{zeroisweaklyclosed}
						Let $\L$ be a 2-pointed 2-category, with a fixed choice $0$ of a representative for the bizero object. The 2-ideal $(0)$ canonically associated to $\L$ is weakly closed. 
					\end{proposition}
					\begin{proof}
						By \prox\ref{charactweaklyreflect}, all 2-kernels weakly reflect null morphisms. Indeed all null morphisms factor through $0$, which is a null object, and it is easy to see that condition (ii) of \prox\ref{charactweaklyreflect} is satisfied. We prove that 2-kernels reflect null 2-cells. So let $\mu$ be a 2-cell as on the left hand side below such that $k\ast \mu$ is a null 2-cell:
						\[\begin{tikzcd}[ampersand replacement=\&,row sep=2ex]
							\& 0 \&\&\&[-3ex]\&[-3ex]\& 0 \& K \\
							D \&\& K \& A \& {=} \& D \&\& {} \& A \\
							\& 0 \&\&\&\&\& 0 \& K
							\arrow["i", from=1-2, to=2-3]
							\arrow["\mu", Rightarrow, from=1-2, to=3-2, shorten <=1.5ex, shorten >=1.5ex]
							\arrow["i", from=1-7, to=1-8]
							\arrow[""{name=0, anchor=center, inner sep=0}, equals, from=1-7, to=3-7]
							\arrow["k", from=1-8, to=2-9]
							\arrow["t", from=2-1, to=1-2]
							\arrow["{t'}"', from=2-1, to=3-2]
							\arrow["k", from=2-3, to=2-4]
							\arrow["t", from=2-6, to=1-7]
							\arrow["{t'}"', from=2-6, to=3-7]
							\arrow["{i'}"', from=3-2, to=2-3]
							\arrow["{i'}"', from=3-7, to=3-8]
							\arrow["k"', from=3-8, to=2-9]
							\arrow["{\exists ! \cong}"{description}, draw=none, from=0, to=2-9]
							\arrow["{\exists ! \cong}"{description}, draw=none, from=2-6, to=0]
						\end{tikzcd}\]
						By uniqueness of the 2-cells $k\c i\aR{} k\c i'\:0\to A$, the pasting on the right hand side above is equal to
						\[\begin{tikzcd}[ampersand replacement=\&, row sep=2ex]
							\& 0 \\
							D \&\& K \& A \\
							\& 0
							\arrow["i", from=1-2, to=2-3]
							\arrow[""{name=0, anchor=center, inner sep=0}, equals, from=1-2, to=3-2]
							\arrow["t", from=2-1, to=1-2]
							\arrow["{t'}"', from=2-1, to=3-2]
							\arrow["k", from=2-3, to=2-4]
							\arrow["{i'}"', from=3-2, to=2-3]
							\arrow["{\exists ! \cong}"{description}, draw=none, from=0, to=2-3]
							\arrow["{\exists ! \cong}"{description}, draw=none, from=2-1, to=0]
						\end{tikzcd}\]
						Since $k$ is faithful, we conclude that $\mu$ is null. Dually for 2-cokernels. 
					\end{proof}
					
					\begin{remark}
						The 2-ideal $(0)$ canonically associated to a 2-pointed 2-category is not always a closed 2-ideal. Moreover, it does not seem to automatically satisfy any of the properties listed in condition $(ii)$ of \thex\ref{theorchargrandistwocat}. With regards to kernels not always being kernels of their cokernels, a counter-example can be found in \cite[\prox{97}]{Dup08}.
					\end{remark}
					
					\begin{definition}
						Let $\L$ be a 2-pointed 2-category, with a fixed choice $0$ of a representative for the bizero object. We say that $0$ is a \emph{strong bizero object} if for every $D\aar{t}0\aar{i} K$ and $D\aar{t'}0\aar{i'} K$ morphisms in $\L$, there is a unique 2-cell
						\[\begin{tikzcd}[ampersand replacement=\&, row sep=1ex]
							\& 0 \\
							D \&\& K \\
							\& 0
							\arrow["i", from=1-2, to=2-3]
							\arrow["{\exists ! }", Rightarrow, from=1-2, to=3-2, shorten <=1.5ex, shorten >=1.5ex]
							\arrow["t", from=2-1, to=1-2]
							\arrow["{t'}"', from=2-1, to=3-2]
							\arrow["{i'}"', from=3-2, to=2-3]
						\end{tikzcd}\]
						In other words, $0$ is a strong bizero object if all 2-cells between null morphisms are null 2-cells.
					\end{definition}
					
					\begin{proposition}
						Let $\L$ be a 2-pointed 2-category with a strong bizero object $0$. Then the 2-ideal $(0)$ canonically associated to $\L$ is a closed 2-ideal.
					\end{proposition}
					\begin{proof}
						The proof is clear after \prox\ref{zeroisweaklyclosed}. 
					\end{proof}
					
					We are now ready to introduce a slightly weaker version of \defx\ref{def2exactness}.
					
					\begin{definition}\label{defweaklygrandispuppe}
						A \emph{weakly Grandis 2-exact} 2-category is a 2-category $\L$ equipped with a 2-ideal $\N$ of null morphisms and null 2-cells such that
						\begin{itemize}
							\item[-] $\L$ has all $\N$-2-kernels and $\N$-2-cokernels;
							\item[-] $\N$ is a weakly closed 2-ideal;
							\item[-] every $\N$-2-kernel $m$ is an $\N$-2-kernel of its $\N$-2-cokernel $c_m$, with structure isomorphic 2-cell between $c_m\c m$ and a null morphism given up to an invertible null 2-cell by the $\N$-2-cokernel $c_m$ of $m$; and dually for $\N$-2-cokernels;
							\item[-] every morphism $f$ in $\L$ factorizes up to isomorphism as an $\N$-2-cokernel followed by an $\N$-2-kernel
							\begin{cd}[0]
								A \ar[rr,iso,shift right=0.75 ex]\ar[rr,bend left=10,"f"]\ar[rd,twoheadrightarrow,"2-\operatorname{coker}"'{pos=0.7}]\&\& B \\
								\& Q \ar[ru,rightarrowtail,"2-\operatorname{ker}"'{pos=0.3}]
							\end{cd}
						\end{itemize}
						
						A \emph{weakly Puppe 2-exact} 2-category is a 2-pointed 2-category $\L$, with a fixed choice $0$ of a representative for a bizero object, that satisfies the conditions above of weakly Grandis 2-exactness with respect to the 2-ideal $(0)$ canonically associated to $\L$. Weakly closedness of the 2-ideal $(0)$ is automatic by \prox\ref{zeroisweaklyclosed}.
					\end{definition}
					
					\begin{remark}
						The only difference between weakly Grandis (respectively, weakly Puppe) 2-exactness and Grandis (respectively, Puppe) 2-exactness is that we only require the 2-ideal to be weakly closed rather than closed. Remember that weakly reflection of null morphisms arose naturally in Section \ref{sectiontwoideals}, when noticing that two equivalent characterizations of closed ideal in dimension 1 are not equivalent anymore in dimension 2.
						
						Notice that, of course, a 2-pointed 2-category with a strong bizero object is weakly Puppe 2-exact precisely when it is Puppe 2-exact.
					\end{remark}
					
					A large part of condition $(i)$ of \thex\ref{theorchargrandistwocat} still holds for weakly Grandis 2-exact 2-categories. One of the most delicate aspects is the orthogonality condition of the factorization system. We will need to weaken such condition as follows.
					
					\begin{definition}\label{ROFS}
						Let $\L$ be a 2-category and let $\N$ be a 2-ideal of null morphisms and null 2-cells. We say that $(\E,\M)$ is a \emph{relatively orthogonal cokernel-kernel factorization system} on $\L$ if the following conditions are satisfied:
						\begin{enum}
							\item every morphism in $\E$ is an $\N$-2-cokernel and every morphism in $\M$ is an $\N$-2-kernel;
							\item every morphism in $\L$ factors up to isomorphism as a morphism in $\E$ followed by a morphism in $\M$;
							\item if $m\in \M$ and $i$ is an equivalence then $m\circ i \in \M$, and dually for $\E$;
							\item if $m\in \M$ and $\alpha: m \Rightarrow f$ is invertible then $f \in \M$, and dually for $\E$;
							\item given a square $(s,t,\phi)$ between $e\in \E$ and $m\in \M$
							\begin{cd}
								{D} \ar[r,"s"]  \ar[d,"e"', two heads]  \ar[r,"\phi"'{inner sep=1.3ex}, iso, shift right=4.5ex]\& {Q}  \ar[d,"m",rightarrowtail]  \\
								{R}  \ar[r,"t"'] \& {A}
							\end{cd}
							where $m$ is the $\N$-2-kernel of $f\:A\to B$ with structure 2-cell $\alpha$ and $e$ is the $\N$-2-cokernel of $g\:C \to D$ with structure 2-cell $\beta$, such that the 2-cell 
							\begin{cd}
								{C} \ar[r, shift right=7ex, "\beta"'{pos=0.8}, iso] \ar[rdd,null, bend right=20] \ar[rrrddd,null, bend right=55] \ar[rrrddd,null, bend left=55] \ar[rd, "g"] \&[-2ex] {} \& {} \& {}\\[-4ex]
								{} \&  {D} \ar[r,"s"] \ar[r, iso, shift left=3ex, "\nu"'{pos=0.8}] \ar[d,"e"', two heads]  \ar[r,"\phi"'{inner sep=1.3ex}, iso, shift right=4.5ex]\& {Q} \ar[rdd, bend left=20, null]   \ar[d,"m",rightarrowtail] \& {} \\
								{} \& {R} \ar[r, iso, shift right=3ex, "\nu"'{pos=0.8}]   \ar[r,"t"'] \& {A} \ar[r, shift left=2ex,xshift=-2ex, "\alpha^{-1}"'{pos=0.8},iso] \ar[rd, "f"']\& {}\\[-4ex]
								{} \& {} \& {} \& {B}
							\end{cd}
							is an invertible null 2-cell, there exist an arrow $d\:R \to Q$ and invertible 2-cells $\sigma$ and $\rho$ as below such that 
							\begin{eqD*}
								\begin{cd}*
									{D} \ar[r,"s"]  \ar[d,"e"', two heads]  \ar[r,"\phi"'{inner sep=1.3ex}, iso, shift right=4.5ex]\& {Q}  \ar[d,"m",rightarrowtail]  \\
									{R}  \ar[r,"t"'] \& {A}
								\end{cd}
								\quad = \quad
								\begin{cd}*
									{D} \ar[r,"s"]  \ar[d,"e"', two heads]  \ar[r,"\sigma"'{inner sep=-0.5ex, pos=0.8}, iso, shift right=2.5ex, xshift=-2ex]\& {Q}  \ar[d,"m",rightarrowtail]  \\
									{R} \ar[r,"\rho"'{inner sep=-0.2ex, pos=0.8}, iso, shift left=2.5ex, xshift=1.2ex]\ar[ru,"d"{inner sep=0.2ex, pos=0.45}] \ar[r,"t"'] \& {A}
								\end{cd}
							\end{eqD*}
							Moreover, given squares $(s,t,\phi)$ and $(s',t',\phi')$, filled by $(d,\sigma,\rho)$ and $(d',\sigma',\rho')$ respectively and a 2-cell $(\lambda,\mu)$ between such squares, there exists a unique 2-cell $\iota\:d\aR{}d'$ such that
							\begin{eqD*}
								\begin{cd}*[6][9]
									\& Q \ar[d,"m",rightarrowtail] \\
									R \ar[r,"t'"']\ar[ru,bend left=17,"d"{inner sep=0.1ex},""{name=A}] \ar[ru,bend right=17,"d'"'{inner sep=0.1ex},""'{name=B}] 
									\ar[r,iso,shift left=2.7ex,"\rho'"'{pos=0.6},start anchor={[xshift=7ex]}]
									\& A
									\ar[from=A,to=B,Rightarrow,"\iota",shorten <=0.5ex,shorten >=0.5ex]
								\end{cd}\quad \h[2] = \quad \h[-3]
								\begin{cd}*[6][9]
									\& Q \ar[d,"m",rightarrowtail] \\
									R \ar[r,"t",bend left=16,""{name=A}]\ar[r,"t'"',bend right=20,""'{name=B}]\ar[ru,bend left=17,"d"{inner sep=0.1ex},""]
									\ar[r,iso,shift left=3.7ex,"\rho"{pos=0.6,inner sep=1ex},start anchor={[xshift=4.5ex]}]
									\& A
									\ar[from=A,to=B,Rightarrow,"\mu",shorten <=0.5ex,shorten >=0.5ex]
								\end{cd}
							\end{eqD*}
							and 
							\begin{eqD*}
								\begin{cd}*[6][9]
									D \ar[d, "e"', two heads] \ar[r,"s"] \ar[r,iso,shift right=2.1ex,"\sigma"'{pos=0.6},start anchor={[xshift=-7ex]}] \& Q \\
									R \ar[ru,bend left=17,"d"{inner sep=0.1ex, pos=0.65},""{name=A}] \ar[ru,bend right=17,"d'"'{inner sep=0.1ex},""'{name=B}] 
									\& 
									\ar[from=A,to=B,Rightarrow,"\iota",shorten <=0.5ex,shorten >=0.5ex]
								\end{cd}\quad \h[2] = \quad \h[-3]
								\begin{cd}*[6][9]
									D \& Q \\
									R
									\arrow[""{name=0, anchor=center, inner sep=0}, "s", bend left=18, from=1-1, to=1-2]
									\arrow[""{name=1, anchor=center, inner sep=0}, "{s'}"', bend right=18, from=1-1, to=1-2]
									\arrow["{\sigma'}"',iso, shift right=7.5, xshift =-2ex,from=1-1, to=1-2]
									\arrow["e"', two heads, from=1-1, to=2-1]
									\arrow["{d'}"', bend right=20, from=2-1, to=1-2]
									\arrow["\lambda", Rightarrow, from=0, to=1,shorten <=0.5ex,shorten >=0.5ex]
								\end{cd}
							\end{eqD*}
						\end{enum}
					\end{definition}
					
					\begin{remark}
						A relatively orthogonal cokernel-kernel factorization system is a factorization system with a weaker orthogonality condition relative to the 2-ideal $\N$. Indeed, condition (v) of Definition \ref{ROFS} is the usual orthogonality required for a 2-dimensional orthogonal factorization system restricted to the squares from $e\in \E$ to $m\in \M$ that are compatible with the structure 2-cells that present $e$ and $m$ as a cokernel and a kernel respectively.
					\end{remark}
					\begin{remark}
						Given a relatively orthogonal cokernel-kernel factorization system $(\E,\M)$ on a 2-category $\L$, we still have the 2-functors $\operatorname{dom}\:\Esq\to \L$ and $\operatorname{cod}\:\Msq\to \L$ introduced in \thex\ref{teorfibofsubobj}, that consider the 2-quotients and the 2-subobjects relative to $(\E,\M)$. However, they might not be weak 2-(op)fibrations anymore.
					\end{remark}
					
					We now show that a large part of the proof of $(ii)\aR{}(i)$ in \thex\ref{theorchargrandistwocat} can still be carried out when we relax the assumption of a closed 2-ideal to that of a weakly closed 2-ideal.
					
					\begin{theorem}
						Let $\L$ be a weakly Grandis 2-exact 2-category, with respect to a weakly closed 2-ideal $\N$. Then the classes $\E$ of all $\N$-2-cokernels and $\M$ of all $\N$-2-kernels form a relatively orthogonal cokernel-kernel factorization system on $\L$ such that there exists a biequivalence over $\L$ \begin{cd}
							\Esq \ar[rd,"\operatorname{dom}"'{inner sep=0.2ex}]\ar[rr,bend left=16,"K"]\ar[rr,simeq]\&\& \Msq \ar[ld,"\operatorname{cod}"{inner sep=0.2ex}]\ar[ll,bend left=16,"C"]\\
							\& \L
						\end{cd}
						between the 2-functors $\operatorname{dom}$ and $\operatorname{cod}$  associated to $(\E,\M)$, with $K$ and $C$ normal pseudofunctors.
						
						Analogously for $\L$ weakly Puppe 2-exact. 
					\end{theorem}
					\begin{proof}
						We start by proving that $(\E.\M)$ is a relatively orthogonal cokernel-kernel factorization system. Condition (i) and (ii) of \defx \ref{ROFS} are satisfied by construction and conditions (iii) and (iv) are straightforward to prove using the properties of $\N$-2-kernels and $\N$-2-cokernels. To prove (iv), consider a square $(s,t,\phi)$ between $e\in \E$ and $m\in \M$ as below, where $m$ is the $\N$-2-kernel of $f\:A\to B$ with structure 2-cell $\alpha$ and $e$ is the $\N$-2-cokernel of $g\:C \to D$ with structure 2-cell $\beta$, such that the 2-cell 
						\begin{cd}
							{C} \ar[r, shift right=7ex, "\beta"'{pos=0.8}, iso] \ar[rdd,null, bend right=20] \ar[rrrddd,null, bend right=55] \ar[rrrddd,null, bend left=55] \ar[rd, "g"] \&[-2ex] {} \& {} \& {}\\[-4ex]
							{} \&  {D} \ar[r,"s"] \ar[r, iso, shift left=3ex, "\nu"'{pos=0.8}] \ar[d,"e"', two heads]  \ar[r,"\phi"'{inner sep=1.3ex}, iso, shift right=4.5ex]\& {Q} \ar[rdd, bend left=20, null]   \ar[d,"m",rightarrowtail] \& {} \\
							{} \& {R} \ar[r, iso, shift right=3ex, "\nu"'{pos=0.8}]   \ar[r,"t"'] \& {A} \ar[r, shift left=2ex,xshift=-2ex, "\alpha^{-1}"'{pos=0.8},iso] \ar[rd, "f"']\& {}\\[-4ex]
							{} \& {} \& {} \& {B}
						\end{cd}
						is an invertible null 2-cell. Notice that given this invertible null 2-cell we can break up the $\nu$ relative to $s\c g$ into the pasting of the one relative to $s$ with the one relative to $g$ up to an isomorphic null 2-cell and so the 2-cell 
						\begin{cd}
							{D} \ar[rrdd, bend left=80, null]  \ar[r,"s"]  \ar[d,"e"', two heads]  \ar[r,"\phi"'{inner sep=1.3ex}, iso, shift right=4.5ex]\& {Q} \ar[r,iso,"\nu"', shift right=4ex, xshift=3ex]\ar[rdd, bend left=20, null]   \ar[d,"m",rightarrowtail] \&[2ex] {} \\
							{R}    \ar[r,"t"'] \& {A} \ar[rd,"f"'] \ar[r, shift left=2ex,xshift=-2ex, "\alpha^{-1}"'{pos=0.8},iso] \& {}\\[-6ex]
							{} \& {} \& {B}
						\end{cd}
						is compatible with $\beta$. Hence, the fact that $e$ weakly coreflects null morphisms induces a 2-cell 
						\[\begin{tikzcd}[ampersand replacement=\&]
							R \& A \& B
							\arrow["t", from=1-1, to=1-2]
							\arrow["\phi"', shift right=3,  iso, from=1-1, to=1-3]
							\arrow[bend right=40, from=1-1, null,to=1-3, "\hat{\phi}"']
							\arrow["f", from=1-2, to=1-3]
						\end{tikzcd}\]
						with the same properties as the one of the proof of \thex \ref{theorchargrandistwocat}. From here on the same argument of the proof of \thex \ref{theorchargrandistwocat} allows us to conclude the 1-dimensional part of the orthogonality property. The proof of the 2-dimensional part of the orthogonality property remains precisely the same as in \thex \ref{theorchargrandistwocat} since we still have that $e$ coreflects null 2-cells. So $(\E.\M)$ is a relatively orthogonal cokernel-kernel factorization system. Moreover, the existence of the biequivalence between $K$ and $C$ follows from the same argument of the proof of \thex\ref{theorchargrandistwocat}, since that part of the proof does not use that $\N$ is a closed ideal (not even the weak closedness of $\N$ is is needed there).
					\end{proof}
					
					We now compare our theory with the 2-dimensional Puppe exactness properties studied by Dupont and Nakaoka. Notice first that all the considered notions of Puppe 2-exactness require the 2-category to have all 2-kernels and 2-cokernels. There is then some variation on which morphisms are required to be 2-kernels of their 2-cokernels and dually 2-cokernels of their 2-kernels.
					
					\begin{remark}\label{remcompareabelian}
						It is almost true that every abelian $\mathrm{Gpd}^\ast$-category is weakly Puppe 2-exact. Indeed, in an abelian $\mathrm{Gpd}^\ast$-category, 0-monomorphisms coincide with 2-kernels, by \cite[{\corx{158}, \prox{166}}]{Dup08}. So the first three out of four conditions of weakly Puppe 2-exactness (see \defx\ref{defweaklygrandispuppe}) are satisfied. Moreover, by \cite[{\corx{160}}]{Dup08}, every abelian $\mathrm{Gpd}^\ast$-category has a factorization system given by the epimorphisms (in Dupont's terminology) and the fully 0-faithful arrows. While epimorphisms are precisely the 2-cokernels in an abelian $\mathrm{Gpd}^\ast$-category, fully 0-faithful arrows are not necessarily 2-kernels. So in general not every abelian $\mathrm{Gpd}^\ast$-category is weakly Puppe 2-exact, nor the converse holds. However, the stronger Nakaoka's notion of relatively exactness also ensures that fully 0-faithful arrows are 2-kernels. We obtain the following theorem.
					\end{remark}
					
					\begin{theorem}
						For $\L$ a locally SCG category, the following are equivalent:
						\begin{enumT}
							\item $\L$ is a relatively exact $\mathrm{Gpd}^\ast$-category;
							\item $\L$ is weakly Puppe 2-exact, every faithful morphism is a 2-kernel and every cofaithful morphism is a 2-cokernel.
						\end{enumT}
					\end{theorem}
					\begin{proof}
						We prove $(i)\aR{} (ii)$. By \remx\ref{remcompareabelian} and \remx\ref{Nakaokacomparison}, we know that every relatively exact $\mathrm{Gpd}^\ast$-category satisfies the first three out of four conditions of weakly Puppe 2-exactness, because every abelian $\mathrm{Gpd}^\ast$-category does so. Moreover, again by \remx\ref{remcompareabelian} and \remx\ref{Nakaokacomparison}, we know that every morphism factorizes up to isomorphism as a 2-cokernel followed by a fully 0-faithful morphism. But fully 0-faithful morphisms are 0-faithful and, by Lemma 5.2 of \cite{Nak10}, in a relatively exact $\mathrm{Gpd}^\ast$-category  0-faithful morphisms are faithful. Moreover by definition of relatively exact $\mathrm{Gpd}^\ast$-category, every faithful morphism is a 2-kernel and thus we conclude that every morphism factorizes up to isomorphism as a 2-cokernel followed by a 2-kernel. 
						
						We now prove $(ii)\aR{} (i)$. By assumption if $\L$ is weakly Puppe 2-exact it has all 2-kernels and all 2-cokernels. Moreover, in a weakly Puppe 2-exact 2-category every 2-kernel is a 2-kernel of its 2-cokernel and so, thanks to the assumption that every faithful monomorphism is a 2-kernel, we conclude that every faithful morphism is a 2-kernel of its 2-cokernel. And dually for cofaithful morphisms.  
					\end{proof}

					\begin{remark}
						In the context of locally SCG categories, the distance between our weakly 2-Puppe exactness and Nakaoka's relatively exactness is thus the request that all faithful morphisms, i.e.\ all 2-dimensional monomorphisms, are 2-kernels and all cofaithful morphisms, i.e.\ all 2-dimensional epimorphisms, are 2-cokernels. In dimension 1, if every morphism factorizes as a cokernel followed by a kernel, then automatically all monomorphisms are kernels and all epimorphisms are cokernels. But the same argument does not extend to dimension 2.
					\end{remark}
					
					\begin{theorem}
						Let $\L$ be a 2-Puppe-exact $\mathrm{Gpd}^\ast$-category (\'{a} la Dupont). Then $\L$ is weakly Puppe 2-exact.
					\end{theorem}
					
					\begin{proof}
						By definition of 2-Puppe-exact $\mathrm{Gpd}^\ast$-category (\'{a} la Dupont), $\L$ has all 2-kernels and all 2-cokernels with respect to the ideal $(0)$. Moreover, the ideal $(0)$ is weakly closed by \prox \ref{zeroisweaklyclosed}. 
						
						We prove that every 2-kernel is a 2-kernel of its 2-cokernel with appropriate structure 2-cell. Let $m$ be a 2-kernel. Thanks to the universal property of the 2-kernel $m$, is faithful and thus it is 0-faithful. But then, by Proposition 179 of \cite{Dup08}, $m$ is canonically the 2-kernel of its 2-cokernel. Analogously, the dual property holds for 2-cokernels.
						
						Finally, the fact that every morphism in $\L$ factorizes up to isomorphism as a 2-cokernel followed by a 2-kernel follows from Proposition 179 of \cite{Dup08}.
					\end{proof}
					
					\begin{remark}
						Our notion of weakly Puppe 2-exactness thus generalizes both Nakaoka's relatively exactness and Dupont's 2-Puppe-exactness, and as a consequence also Dupont's 2-abelianness. So all the examples of relatively exact $\mathrm{Gpd}^\ast$-category studied in \cite{Nak08} and all the examples of 2-Puppe-exact (or 2-abelian) $\mathrm{Gpd}^\ast$-category (\'{a} la Dupont) studied in \cite{Dup08} are examples of our theory.
					\end{remark}
					
					The following diagram condenses the comparison between the considered 2-dimensional notions of Puppe exactness:
					\begin{cd}[1]
						\&\& {\underset{\mathrm{Gpd}^\ast\text{-category}}{\text{abelian}}} \\
						\& {\underset{\mathrm{Gpd}^\ast\text{-category}}{\text{relatively exact}}} \\
						{\underset{\mathrm{Gpd}^\ast\text{-category}}{\text{2-abelian}}} \&\& {\text{weakly Puppe 2-exact}} \\
						\& {\underset{\text{(\'{a} la Dupont)}}{\underset{\mathrm{Gpd}^\ast\text{-category}}{\text{2-Puppe-exact}}}}
						\arrow[Rightarrow, from=2-2, to=1-3]
						\arrow[Rightarrow, from=2-2, to=3-3]
						\arrow[Rightarrow, from=3-1, to=2-2]
						\arrow[Rightarrow, from=3-1, to=4-2]
						\arrow[Rightarrow, from=4-2, to=3-3]
					\end{cd}

					Our work applies the theory of (2-dimensional) fibrations, as well as that of factorization systems, to abstractly deduce how to define Grandis and Puppe exactness in dimension 2, as opposed to trying to extract the right conditions from inspecting the examples. This paper also sheds new light on the 2-dimensional cohomology theory started by Dupont and Nakaoka.
					
					

				\end{document}